\documentclass[11pt,american]{amsart}
\usepackage{amsmath,amsfonts,amssymb,amscd,amsthm,amsbsy,bbm,epsf,calc,graphicx,esint}
\usepackage{mathtools}
\usepackage{color}
\numberwithin{equation}{section}

\newtheorem{theorem}{Theorem}[section]

\newtheorem{lemma}[theorem]{Lemma}

\theoremstyle{definition}

\newtheorem{remark}[theorem]{Remark}

\newcommand{\R}{\mathbb R}

\newcommand{\dd}{\, \mathrm{d}}

\newcommand{\tr}{\mbox{tr}}

\newcommand{\norm}[1]{\left\lVert#1\right\rVert}
\newcommand{\abs}[1]{\left\vert#1\right\vert}
\newcommand{\set}[1]{\left\{#1\right\}}

\numberwithin{equation}{section}
\title[Non-existence of some blowup profiles]{Non-existence of some approximately self-similar singularities for the Landau, Vlasov-Poisson-Landau, and Boltzmann equations}

\author{Jacob Bedrossian}
\address{Department of Mathematics, University of Maryland, College Park, MD 20742 USA}
\email{jacob@cscamm.umd.edu} 
\thanks{JB was partially supported by National Science Foundation RNMS \#1107444 (Ki-Net)}
\author{Maria Pia Gualdani}
\address{Department of Mathematics, The University of Texas at Austin, 2515 Speedway, Austin TX, 78712}
\email{gualdani@math.utexas.edu}
\thanks{MG is partially supported by the DMS-NSF 2019335 and would like to thanks NCTS Mathematical division Taipei for their kind hospitality.}

\author{Stanley Snelson}
\address{Department of Mathematical Sciences, Florida Institute of Technology, Melbourne, FL 32901}
\email{ssnelson@fit.edu}

\thanks{SS was partially supported by a Ralph E. Powe Award from ORAU}
\thanks{This project began at the AIM Workshop ``Nonlocal differential equations in collective behavior.'' The authors would like to thank AIM for their kind hospitality.}
\begin{document}

\maketitle

\begin{abstract}
We consider the homogeneous and inhomogeneous Landau equation for very soft and Coulomb potentials and show that approximate Type I self-similar blow-up solutions do not exist under mild decay assumptions on the profile. We extend our analysis to the Vlasov-Poisson-Landau system and to the Boltzmann equation without angular cut-off.

\end{abstract}

\section{introduction}

We consider the inhomogeneous Landau equation: 
\begin{equation}\label{e:mainL}
\partial_t f + v\cdot \nabla_x f = Q_L(f,f) := \tr( \bar a^f D_v^2 f) + \bar c^f f,
\end{equation}
where, for $f:\R_+\times \R^3 \times \R^3\to \R$, 
\begin{equation}\label{e:coeffs}
\begin{split}
\bar a^f(t,x,v) &:= a_{\gamma}\int_{\R^3} \Pi(v_*) |v_*|^{\gamma+2} f(t,x,v-v_*)\dd v_*,\\
\bar c^f(t,x,v) &:= \begin{cases} c_{\gamma} \int_{\R^3} |v_*|^{\gamma} f(t,x,v-v_*)\dd v_*, &\gamma >-3,\\
f, &\gamma = -3,\end{cases}
\end{split}
\end{equation}
and $\Pi(z) := \left(Id - \dfrac {z\otimes z}{|z|^2}\right)$. The constants $a_\gamma$ and $c_\gamma$ are positive and only depend on $\gamma$. The constant $\gamma$ belongs to the range of very soft potentials, i.e. $\gamma \in [-3,-2]$. For $\gamma\le -2$ the Landau collision operator $Q_L$ shares several similarities with the semilinear operator $\Delta f + f^2$ and a question naturally arises: do smooth solutions  to  (\ref{e:mainL}) stay bounded for all times or do they become unbounded after a finite time? We say that a solution $f$ blows up at a time $T<+\infty$ if it is well defined for all $0<t<T$, and if 
$$
\lim_{t\to T^-} \|f(t,x,v)\|_{L^{\infty}( \R^3 \times \R^3)}=+\infty.
$$
We would call $T$ the blow-up time for $f$. This question of regularity versus singularity formation for (\ref{e:mainL}) is, at the present day, still unanswered.

%Let us briefly discuss previous work on the question of global existence vs. breakdown for the Landau and Boltzmann equations. 

The existence of smooth solutions to the inhomogeneous Landau equation (\ref{e:mainL}) for very soft potentials is known for a short time,  \cite{he2014boltzmannlandau, HST2019rough, HST2018landau}, and for long times under 
simplifying assumptions on the initial data. For example, when the initial data is sufficiently close to
a Maxwellian equilibrium state, solutions exist globally and converge to equilibrium \cite{guo2002periodic}. %, carrapatoso2016cauchy}. %\cite{lions1994boltzmannlandau}    \cite{S2018hardpotentials} and converge to equilibrium . 
Solutions are also known to exist when initial data are near vacuum in the cases of moderately soft potentials \cite{luk2019vacuum} and hard potentials \cite{chaturvedi2020vacuum}. Recently, several new studies concerning regularity and continuation criteria have appeared; these results are based on {\em conditional} assumptions on the hydrodynamic quantities, see \cite{golse2016, cameron2017landau, henderson2017smoothing,  chen2009smoothing, liu2014regularization}.
The situation for Vlasov-Poisson-Landau is less well-studied; see \cite{Guo2012,strain2013vlasov,chaturvedi2021vlasov} for results near global Maxwellians and \cite{duan2020vlasov} for results near some local Maxwellian data. 

The available literature on the homogeneous version of (\ref{e:mainL}) for very soft potentials is larger. In \cite{arsenev-peskov}, \cite{villani1996global}, \cite{alexandre2004landau} and, later, in \cite{desvillettes2015landau}, the authors show global existence of weak solutions. Recently, it was proven that for short time weak solutions become instantaneously regular and smooth, see \cite{silvestre2015landau} and \cite{gualdani2017landau}. The question of whether they stay smooth for all time or become unbounded after a finite time is, however, also in this case, still open. Recent research has also produced several conditional results. This includes uniqueness results in \cite{Fournier2010} for solutions that belong to the space $L^1(0,T,L^\infty(\mathbb{R}^3))$ and in \cite{ChGu20} for solutions in $L^\infty(0,T,L^p(\mathbb{R}^3))$ with $p>\frac{3}{2}$; and regularity results for solutions in  $L^\infty(0,T,L^p(\mathbb{R}^d))$ with $p>\frac{d}{2}$ \cite{silvestre2015landau} \cite{gualdani2017landau}. We also mention the long time asymptotic results for  weak solutions from \cite{CM2017verysoft} and \cite{CDH15}.

In the very recent manuscript \cite{Desvillettes-He-Jiang-2021} the authors studied behavior of solutions in the space $L^\infty(0,T,\dot{H}(\mathbb{R}{^3}))$. They show that for general initial data there exists a time $T^*$ after which the weak solution belongs to $L^\infty((T^*, +\infty), {H^1}(\mathbb{R}{^3}))$. This result is in accordance with the one  in \cite{GGIV2019partial}; in \cite{GGIV2019partial} the authors  showed that the set of singular times for weak solutions has Hausdorff measure at most $\frac{1}{2}$. On the other hand, global existence of bounded smooth solutions has been shown for an {\em{isotropic}} modification of the Landau  equation $\partial_t f= tr(\bar a^f  \Delta f) + f^2$ in \cite{gualdani2014radial}.

For the non-cutoff Boltzmann equation, existence theory is in a roughly similar stage as the Landau equation. Global existence is known for initial data that is close to equilibrium \cite{gressman2011boltzmann}. %( see also \cite{amuxy2011global, alexandre2012global} . 
In the space homogeneous case, solutions are known to exist globally when $\gamma +2s\geq 0$ \cite{he2012homogeneous-boltzmann}. (The parameter $s$ will be defined in Section \ref{s:boltzmann} below.) See also 
\cite{lu2012measure, morimoto2016measure} for global existence of measure-valued solutions in the homogeneous setting, which regularize in some cases. Short-time existence for the inhomogeneous equation was established in various regimes in, e.g. \cite{amuxy2010regularizing, amuxy2011bounded, HST2019boltzmann}. There is also a program of conditional regularity (see \cite{silvestre2016boltzmann,imbert2016weak, imbert2018decay, imbert2019smooth}) that gives $C^\infty$ smoothness in the case $\gamma + 2s \in [0,2]$, provided the mass, energy, and entropy densities remain under control.

%\textcolor{red}{[Review of self-similar singularity literature: Type I and II self-similar blow-up in semilinear heat equation and Keller-Segel. then results and conjectures about other equations with a degenerate scaling symmetry (what is the right term?), specifically Vlasov (in particular, note the Type I blow-up of relativistic Vlasov by Raphael etc), Burgers equation (I think Masmoudi, Ghoul, and collaborators have a paper on this), incompressible Euler (numerical evidence of Hou/Luo, singularity results of Tarek etc), De Grigorio (will have to check state of the art, Tom Hou has some results), CKY model for Euler blow-up. However, one huuuge difference that Landau has over these other models with messed up scalings is the presence of (hypoelliptic) parabolic smoothing. for this reason, it seems unlikely that there could exist low-regularity singularities like in Burgers and Euler, however, due to the nonlinearity, it isn't obvious]}
	%\textcolor{red}{[To check: is compressible Euler similar? -- plane shock formation, point shock, implosion shocks etc]}

As the question of whether or not solutions to the Landau and Boltzmann equations exhibit finite-time singularities remains an open challenge, it is natural to narrow down the search to certain kinds of singularities.  Our goal is to investigate, and {\em eliminate the existence of}, one particular breakdown mechanism, which is usually called \emph{approximately self-similar blowup}. 
Self-similar singularities are very common in nonlinear partial differential equations and can come in many different forms; see \cite{eggers2015singularities,barenblatt1996scaling} for many examples and detailed discussions.
Here we consider a singularity to be (approximately) \emph{self-similar} if the solution is of the following form
\begin{align}
f(t,x) = \frac{1}{\mu(t)} g\left(\frac{x}{\lambda(t)} \right) + \mathcal{E}(t,x), 
\end{align}
where $\mathcal{E}$ is some error (possibly zero) which is less singular than the self-similar term and where $\mu(t),\lambda(t)$ are rates such that $\lim_{t \nearrow T} \mu(t) = 0$
and $\lim_{t \nearrow T} \lambda(t) = 0$. The function $g$ is called the ``(inner) profile''. 
In the literature, self-similar singularities are roughly divided into two kinds (terminology dating from at least \cite{BZ72}): Type I self-similar, in which the blow-up rate is determined by dimensional analysis (i.e. the scaling symmetry of the equations); and Type II self-similar, in which the rate is determined also by other additional effects, for example by an eigenvalue problem associated to the inner blow-up profile. 
In this context, two well-studied equations are the semilinear heat equation and the Keller-Segel system. 
The semilinear heat equation, despite its simplicity, displays both types of singularities; see reviews in \cite{matano2004nonexistence,quittner2019superlinear,C17} and the references therein. 
The Keller-Segel equations displays type II self-similar finite time and infinite time singularities \cite{CGMN19,GM16,RS14}, and with nonlinear diffusion, can display type I self-similar singularities \cite{BL09}. 
Another semilinear parabolic PDE studied in this context are the incompressible Navier-Stokes equations; finite-energy type I self-similar solutions were ruled out in \cite{NRS96,Tsai1998}; see also \cite{chae2007,SS09}. Self-similar singularities are also intensely studied in the setting of dispersive equations, such as for example, the nonlinear Schr\"odinger equations \cite{MR05} and wave equations

One significant difference between the Landau and Boltzmann equations and the semilinear equations discussed above (and many quasilinear problems too) is a two-parameter scaling symmetry. 
That is, if $f$ is a solution to either (\ref{e:mainL}), or (\ref{e:mainB}), then for any $\alpha \in \R$ and $\lambda>0$, so is 
\[f_{\lambda,\alpha}(t,x,v) := \lambda^{\alpha + 3+\gamma} f(\lambda^\alpha t, \lambda^{1+\alpha}x, \lambda v).\]
This provides for a much wider and subtle class of potential type I singularities (and likely type II as well) than equations with a one-parameter scaling symmetry. 
These kinds of two-parameter symmetry groups are common in fluid mechanics and kinetic theory. Some examples include the Burgers equation, which undergoes type I self-similar shock formation \cite{CGM18,eggers2015singularities}, and the isentropic, compressible Euler equations, for which there are many self-similar finite-time singularities, including implosions \cite{MRRS19I,MRRS19II} and shocks \cite{BSV19,buckmaster2019formation,christodoulou2007formation,christodoulou2014compressible}. 
Another example are the incompressible Euler equations, for which the existence or smooth finite-time singular solutions remains open. 
Type I self-similar singularities have been ruled out under a variety of decay and/or integrability conditions on the profile \cite{chae2007,CS13,CW18,ChaeTsai15}, nevertheless, there is strong numerical evidence that type I self-similar singularity formation is possible, at least along the boundary \cite{luo2014potentially}. 
Moreover, in H\"older regularity (as opposed to smooth), there does exist type I self-similar finite-time blow-up solutions \cite{elgindi2019finite,EJ19}. Smooth type I self-similar blow-up solutions have also been constructed for some toy models of the Euler equations, such as the Choi-Kiselev-Yao (CKY) model \cite{hou2015self} and the de Gregorio model \cite{CHD19}. 
Finally, the four-dimensional gravitational Vlasov-Poisson equations have a family of type I self-similar finite-time singularities \cite{LMR08I,LMR08II}. 
One significant difference that Landau and Boltzmann equations (with singular cut-off) from all of the examples just discussed is the presence of hypoelliptic (or parabolic, if homogeneous) smoothing. 
For example, this is likely to rule out the kind of regularity-dependent blow-up dynamics observed in Burgers and Euler \cite{CGM18,elgindi2019finite,EJ19}. 

In light of the rich number and types of blow-up profiles found in similar equations, it makes sense to narrow down the search for potential singularities by eliminating them one at the time. This work can be considered a first study in this direction, endeavoring to rule out as many kinds of Type I singularities as possible. 
Our first main result is summarized in the following statement, which will be presented and discussed in detail in the next section: 

{\bf{Main Theorem Summary}} {\em{Let $\gamma \in [-3,-2]$ and let $f$ be a smooth solutions to \eqref{e:mainL} with mass and kinetic energy locally bounded, namely 
\begin{align*}%\label{e:f-condition}
f \geq 0,\;f \in C^\infty((0,T) \times \R^3_x \times \R^3_v), \quad  \forall R>0, \sup_{0 < t < T} \int_{\abs{x} \leq R}\int \left(1 + \abs{v}^2\right) f \dd v < \infty,
\end{align*}
for any $T>0$. 
Then, if $f$ has the form 
\begin{equation}\label{e:ansatz}
f(t,x,v) = \phi(t,x,v) + \frac 1 {(T-t)^{1+\theta(3+\gamma)}}\, g\left( \frac x {(T-t)^{1+\theta}}, \frac v {(T-t)^{\theta}}\right),
\end{equation}
with $-1 < \theta < 1/2$, $g$ smooth, $\phi$ not too singular as $t \nearrow T$, and $g$ bounded and satisfying mild decay conditions, then we must have $g\equiv 0$. }}\\

{\bf{Corollary.}} {\em{Let $\gamma \in [-3,-2]$ and let $f$ be a smooth solutions to the homogeneous Landau equation 
$$
\partial_t f  =\tr( \bar a^f D_v^2 f) + \bar c^f f.
$$
Then, if $f$ has the form 
\begin{equation*}%\label{e:ansatz_h}
f(t,v) = \phi(t,v) + \frac 1 {(T-t)^{1+\theta(3+\gamma)}}\, g\left(\frac v {(T-t)^{\theta}}\right),
\end{equation*}
with $1/|\gamma| < \theta < 1/2$, $g$ smooth, $\phi$ not too singular as $t \nearrow T$, and $g$ bounded and satisfying mild decay conditions, then we must have $g\equiv 0$. }}\\

In the second part of our manuscript we extend our blow-up analysis to the Vlasov-Poisson-Landau system ($\gamma =-3$):
\begin{equation}\label{e:LCP}
\begin{split}
&\partial_t f + v\cdot \nabla_x f -\nabla_x E \cdot\nabla_v f = Q_L(f,f),\\
& -\Delta_x E = \pm 4\pi \int_{\R^3} f(t,x,v) \dd v,
 \end{split}
\end{equation}
and to the non-cutoff Boltzmann equation: 
\begin{equation}\label{e:mainB}
\partial_t f + v\cdot \nabla_x f = Q_B(f,f) := \int_{\R^3} \int_{\mathbb S^{2}} B(v-v_*,\sigma) [f(v_*')f(v') - f(v_*)f(v)]\dd \sigma \dd v_*,\\
\end{equation}
(See Section \ref{s:boltzmann} for the definitions of $B(v-v_*,\sigma)$, $v'$, and $v_*'$.) For both models we similarly rule out existence of solutions of the form (\ref{e:ansatz}), see Theorem \ref{t:boltzmann} and Theorem \ref{main_VLP}. \\

Let us briefly comment on the admissible values of $\theta$. Define the self-similar variables
\begin{equation*} %\label{e:self}
 y = \frac x {(T-t)^{1+\theta}}, \quad w = \frac v {(T-t)^\theta}.
\end{equation*} 
In these variables, our ansatz becomes %We will think of $g$ as a function of $y$ and $w$.
\begin{equation*}%\label{e:ansatz2}
 f(t,x,v) = \phi\left(t,(T-t)^{1+\theta}y, (T-t)^\theta w\right) + \frac 1 {(T-t)^{1+\theta(3+\gamma)}} g(y,w).
 \end{equation*}
We consider all $\theta$ that satisfy at the same time $1+\theta >0$ and $1+\theta(3+\gamma) > 0$ for all $\gamma \in [-3,-2]$, i.e. we want a solution that forms a singularity a point in space. This implies $\theta >-1$. 
In the case of the homogeneous Landau equations we additionally have the requirement $1/\abs{\gamma} < \theta$ because otherwise, our ansatz violates conservation of mass and is therefore not an admissible solution. %we know in this case that $f \in L^\infty(0,T;L^1)$. 

To motivate the upper bound on $\theta$ that appears in the our results, we recall from \cite{gualdani2017landau, silvestre2015landau}, that if $f$ is a solution to the homogeneous Landau equations which belongs to $L^\infty(0,T,L_v^q(\mathbb{R}^3))$ for some $q > 3/(5+\gamma)$ then $f$ is uniformly bounded. 
Hence, it is natural to require blow up in all of $L_v^q(\mathbb{R}^3))$ with  $ 3/(5+\gamma) < q \le +\infty$ at  $x = 0$. This motivates the requirement $\theta <\frac{1}{2}$, which also appears in the proof in order to control error terms coming from the interaction of $\phi$ and $g$ near the singularity. 

For the Boltzmann equation, we will take the same ansatz, with $\theta > -1$ for the same reasons mentioned above. The upper restriction on $\theta$ that arises from our proof in the Boltzmann case is $1/(2s)$ rather than $1/2$. Note that $2s$ is the order of the diffusion generated by the collision operator, whereas the Landau collision operator gives rise to diffusion of order $2$. For homogeneous Boltzmann, conservation of mass also holds, which rules out values of $\theta$ smaller than $1/|\gamma|$ in our ansatz.

\begin{remark}
Note that $\theta < 0$ and $\theta > 0$ each correspond to very qualitatively different blow-up scenarios.  In $\theta > 0$ the distribution function is forming a singularity at $v=0$; that is, many particles are slowing to a halt near the singularity. For $\theta < 0$, many particles are being accelerated to unbounded velocities near the point of singularity.  Due to the conservation of energy, the latter kind of singularity cannot occur in the homogeneous equations. However there is, a priori, no reason why such a singularity cannot occur in the inhomogeneous equations, such as in (\ref{e:LCP}) or (\ref{e:mainL}). In fact, precisely this kind of approximately Type I self-similar singularity with accelerating particles occurs in the 4-dimensional gravitational Vlasov equation \cite{LMR08I,LMR08II}.  	
\end{remark}

\section{Main results on the Landau equation}\label{Main section}

To properly formulate our results, we need to choose a  proper class of solutions. In that class, we will show that breakdown mechanisms of the form \eqref{e:ansatz} will not occur. 
The conditions we impose on $\phi$ and $g$ in \eqref{e:ansatz} are mild, but somewhat tedious to state. For convenience, by shifting time, we will take $t=0$ to be the potential blowup time, and assume that $f$ is defined for $(t,x,v) \in(-T,0)\times \mathbb R^3 \times \R^3$ for some $T>0$. Hence, we write
\begin{equation}\label{e:ansatzII}
f(t,x,v) = \phi(t,x,v) + \frac 1 {(-t)^{1+\theta(3+\gamma)}}\, g\left( \frac x {(-t)^{1+\theta}}, \frac v {(-t)^{\theta}}\right),
\end{equation}
or 
%ansatz \eqref{e:ansatzII} becomes %We will think of $g$ as a function of $y$ and $w$.
\begin{equation*}%\label{e:ansatz2}
 f(t,x,v) = \phi\left(t,(-t)^{1+\theta}y, (-t)^\theta w\right) + \frac 1 {(-t)^{1+\theta(3+\gamma)}} g(y,w),
 \end{equation*}
 in the self-similar variables
\begin{equation}\label{e:self}
 y := \frac x {(-t)^{1+\theta}}, \quad w := \frac v {(-t)^\theta}.
\end{equation}

The first condition is that the mass and kinetic energy of $f$ are locally bounded,  
\begin{align}\label{e:f-condition}
f \geq 0,  \quad \forall R > 0, \quad \sup_{-T < t < 0}  \int_{\abs{x} < R} \int \left(1 + \abs{v}^2\right) f \dd v \dd x < \infty, 
\end{align}	
in particular, we do not require the solution to decay as $x \to \infty$. Our analysis is, therefore,  valid also for periodic domains and homogeneous solutions (that is, $x$-independent solutions).  

The second condition is that the singularity occurs only at the blow-up space-time point $(t,x)=(0,0)$:
\begin{align}\label{e:singularity}
 f \in C^\infty((-T,0] \times \R^3 \times \R^3 \setminus \set{0} \times \set{0}\times \R^3 ). 
\end{align}

The third condition concerns the inner profile $g$. For all $1 \leq p \leq \infty$, $0\leq j\leq 2$, $0\leq \ell \leq 1$, we require
\begin{eqnarray}\label{e:g-smooth}
 \begin{array}{cc}
              g \in C^\infty(\R^3 \times \R^3), \quad 
              D_w^jD_y^\ell g \in L^{\infty}_{y, \rm loc} L^{p}_{w}.
\end{array} 
\end{eqnarray}

\begin{remark}
	In fact, one can use the slightly weaker condition $g \in L^\infty_{y,loc} L^1_w \cap L^\infty_{y,loc} L^p_w \cap C^\infty$ for some $p > \frac{3}{\gamma + 5}$, however for simplicity of exposition we will use the stronger assumption \ref{e:g-smooth}.  
\end{remark}

%Regarding the behavior of $g$ for large $|y|$, we consider two regimes: $g\in L^1_{y,w}(\R^6)$, and $g=g(w)$ constant in $y$. The second regime includes as a subcase the spatially homogeneous Landau equation, in which $f$, $\phi$, and $g$ do not depend on $x$ at all.

The fourth condition will assure that, near the singularity, the contribution of $\phi$ is small compared to $g$ in the natural self-similar frame. In this regard, the function $\phi$ is allowed to form a singularity at a rate that is `sub-critical' with respect to the scaling. 
First note that 
$$
D_v^jD_x^{\ell} f (v,x,t) = D_v^jD_x^{\ell}  \phi + (-t)^{-1-\theta(3+\gamma)-\ell(1+\theta)-j\theta}D_w^jD_y^{\ell} g(y,w) .
$$
We would like to compare 
$$\sup_{|y|\le R} \| D_v^jD_x^{\ell}  \phi \|_{L^p_v}$$ 
with (using the definition of the self-similar coordinates \eqref{e:self})
$$ (-t)^{-1-\theta(3+\gamma)-\ell(1+\theta)-j\theta}\sup_{|y|\le R} \| D_w^jD_y^{\ell} g \|_{L^p_v} = (-t)^{-1-\theta(3+\gamma)+3\theta /p-\ell(1+\theta)-j\theta}\sup_{|y|\le R} \| D_w^jD_y^{\ell} g \|_{L^p_w}.
$$
 Let's consider first $\ell=j=0$;
 %Since 
%$$
 %\| D_w^jD_y^{\ell} g \|_{L^p_v} = (-t)^{3\theta /p} \| D_w^jD_y^{\ell} g \|_{L^p_w}
%$$
we compare the terms 
$$\sup_{|y|\le R} \|  \phi \|_{L^p_v} \quad \textrm{vs}\quad (-t)^{-1-\theta(3+\gamma)+3\theta /p}\sup_{|y|\le R} \| g \|_{L^p_w}.$$
If $\theta$ and $p$ are such that $-1-\theta(3+\gamma)+3\theta /p <0$,  we enforce $\phi$ to satisfy 
\begin{align}\label{case2}
\lim_{t \to 0} (-t)^{1+\theta(3+\gamma)-3\theta /p} \sup_{|y|\le R} \|  \phi \|_{L^p_v}=0.
\end{align}
In this case, $\sup_{|y|\le 1} \|  f \|_{L^p_v}$ blows up with a rate $(-t)^{-1-\theta(3+\gamma)+3\theta /p}$ dictated by $g$ and the $\phi$ contribution is at least slightly less singular.   
This case happens for all $p\ge 1$ for $\theta < \frac{1}{|\gamma|}$ and for $p>\frac{3\theta}{1+\theta(3+\gamma)}$ if $\theta \ge\frac{1}{|\gamma|}$. Note that this condition with $p=\infty$, implies $g\geq 0$, by taking $t \to 0$.

Generalizing to derivatives, we assume that for all $1 \leq p \leq \infty$ such that $1 + \theta(3+\gamma) - \tfrac{3\theta}{p} \geq 0$, for every $R > 0$, $0\leq i\leq 1$, $0 \leq j \leq 2$, $0 \leq \ell \leq 1$, the function $\phi$ satisfies
\begin{equation} \label{e:lim-phi-t1}
\lim_{t \to 0} (-t)^{1+ \theta(3+\gamma) - \tfrac{3\theta}{p} + i + (1+\theta)\ell + \theta j} \sup_{\abs{y} \leq R} \norm{\partial_t^i D_v^j D_x^\ell \phi(t,(-t)^{1+\theta}y,\cdot)}_{L^p_v} = 0.
\end{equation}
%Finally, we impose that if $1 \leq p < \infty$ and $\theta$ are such that $1 + \theta(3+\gamma) - \tfrac{3\theta}{p} \leq 0$, we assume that for every $0\leq i \leq 1$, $0 \leq j \leq 2$, $0 \leq \ell \leq 1$,
%\begin{align}
%  \sup_{t \in (-T,0)} (-t)^{i+(1+\theta)\ell + \theta j} \sup_{\abs{y} \leq R} \norm{\partial_t^i D_v^j D_x^\ell \phi(t,(-t)^{1+\theta}y,\cdot)}_{L^p_v} < \infty, \label{e:bd-phi-t1}
%\end{align}
%for some $R > 0$. 

Our main result for the Landau equation is summarized in the following theorem:
\begin{theorem}\label{t:landau}
Let $\gamma \in [-3,-2]$, {$-1 < \theta < 1/2$}.  Let $f$ be a smooth solution of the Landau equation \eqref{e:mainL} that satisfies \eqref{e:f-condition} and \eqref{e:singularity}.
Assume that $\phi$ satisfies \eqref{e:lim-phi-t1}. % or \eqref{e:bd-phi-t1}. 
For $g$, assume it satisfies  \eqref{e:g-smooth} and that there exist $h$ and $q$ such that 
\begin{align*}
g(y,w) = q(w) + h(y,w),
\end{align*}
with 
$$(1+ \abs{y} + |w|)h\in L_{y,w}^1(\R^6)\quad \textrm{and} \quad q \in L_w^1(\R^3).$$ % and $h$  and $q$ both satisfy \eqref{e:g-smooth}. \\
Finally,  if $\theta = \pm1/3$ we additionally assume that 
$$(1+ \abs{y}\abs{w}^2 + |w|^3) h \in L^1_{y,w}(\R^6) \quad \textrm{and} \quad (1+ |w|^2) q \in L^1_w(\R^3).$$
	%\item[(b)] if $\theta = 1/3$, then $q$ satisfies the entropy conditions \eqref{e:entropy}.
%\end{enumerate}
Then, for any solution to the Landau equation \eqref{e:mainL} of the form 
\begin{equation*}
f(t,x,v) = \phi(t,x,v) + \frac 1 {(-t)^{1+\theta(3+\gamma)}}\, g\left( \frac x {(-t)^{1+\theta}}, \frac v {(-t)^{\theta}}\right),
\end{equation*}
 we must have $g\equiv 0$ and hence no approximate self-similar singularity of this type can occur.  
\end{theorem}
\begin{remark}
Note that the inhomogeneous problem could have a self-similar profile with $q \neq 0$. 
In other equations, there are type I self-similar singularities with inner profiles that do not decay at infinity (although the solution does), such as in the semilinear heat equation \cite{GigaKohn85}, and even profiles grow at infinity, such as in shock formation in Burgers \cite{CGM18,eggers2015singularities} and in the CKY model \cite{hou2015self}. Numerical evidence suggests that such singularities exist also in the incompressible Euler equations \cite{luo2014potentially}. 
At the current moment, we do not know how to classify potential singularities with inner profiles that grow at infinity. 
\end{remark}
\begin{remark}
If one knows a priori that $g \in L^{1}_{y,v}$, then it suffices to use $(1 + |w|)g\in L_{y,w}^1(\R^6)$ (and $(1+|w|^3) h$ if $\theta = \pm 1/3$).
\end{remark}

Next, we specialize our analysis to the homogeneous Landau equation. Our next theorem shows essentially that if any solution to the homogeneous Landau equation has a singularity, such singularity is  either (i) not of Type I self-similar, or (ii) is of Type I self-similar with a profile $g \not\in L^1(\R^3)$.  

\begin{theorem}\label{theom:Landau_Hom}
 Let $\gamma \in [-3,-2)$ and $1/\abs{\gamma} \leq \theta < 1/2$.  Assume that $f = f(v,t)$ has finite mass and second moment and satisfies 
\begin{align*}%\label{e:singularity}
 f \in C^\infty((-T,0) \times  \R^3), f \in C^\infty((-T,0] \times \R^3 ).
\end{align*}%\eqref{e:singularity}.
Assume that $\phi$ satisfies \eqref{e:lim-phi-t1}, and that $g = g(v)$ is such that 
$$
g\in C^{\infty}(\R^3), \quad g \in L^1_w, %, \quad D^j_{w,loc} g \in L^p_w,
$$
%for all $p\ge 1$ and $0\le j \le 2$. 
If $\theta = 1/3$, then additionally assume that $g$ satisfies $ (1+|w|^2)g \in L^1_w$. 

Then, for any solution to the homogeneous Landau equation of the form 
\begin{equation*}
f(t,v) = \phi(t,v) + \frac 1 {(-t)^{1+\theta(3+\gamma)}}\, g\left( \frac v {(-t)^{\theta}}\right),
\end{equation*}
we must have $g \equiv 0$, and hence no approximate self-similar singularity of this type can occur. 
\end{theorem}

The outline of the paper is as follows: in Section \ref{s:landau} we prove Theorem \ref{t:landau} and Theorem \ref{theom:Landau_Hom}, in Section \ref{s:landau-poisson} we investigate the Vlasov-Landau-Poisson system and in Section \ref{s:boltzmann} the Boltzmann equation. 

\subsection{Notation} 
We will employ the notation $\langle \cdot \rangle = \sqrt{1+|\cdot|^2}$ throughout. When we say $t\to 0$, we always mean that $t$ increases to $0$ through negative values. We will write $A\lesssim B$ when $A\leq CB$ for some universal constant $C$. When integrals appear with no domain, it is assumed that the domain of integration is $\R^3$. Similarly, norms such as $\|\cdot\|_{L^p}$ are over $\R^3$, unless stated otherwise. %For any matrix $A$, we will write $|A| = \sup_{|x|=1} |Ax|$.

\section{Proof of Theorem \ref{t:landau}} \label{s:landau}

%Under our assumptions, the formal identity
%\begin{equation}
%\int_{\R^3} Q(f,f) \dd v = 0
%\end{equation}
%is valid for both $Q = Q_B$ and $Q=Q_L$. This implies that the evolution of \eqref{e:mainB} or \eqref{e:mainL} conserves the total mass $\|f(t,\cdot,\cdot)\|_{L^1_{x,v}}$ (if these quantities are initially finite).  

\subsection{Preliminary lemmas}
First, we recall important global estimates on the coefficients in \eqref{e:mainL}. The proof is standard, but we include a sketch for the readers' convenience. 
\begin{lemma}\label{l:coeffs}
Let $\gamma \in [-3,-2]$. With $\bar a^h$ and $\bar c^h$ defined as in \eqref{e:coeffs}, for any $1 \leq p < \frac{-3}{\gamma + 2}$, there exists $C > 0$ such that %for any small $\delta >0$, $\exists \zeta,\eta,\beta$ \textcolor{red}{(easily computed)} such that 
\begin{align}
|\bar a^h(v)| &\leq C\|h\|_{L^1(\R^3)}^{1 + \frac{p}{3}(\gamma+2)}\|h\|_{L^{\frac{p}{p-1}}(\R^3)}^{-(\gamma+2)p/3}. \label{ineq:a1hi} 
\end{align}
Moreover, for any $1 \leq q < \frac{3}{\gamma + 5}$, there exists $C > 0$ such that 
\begin{align}
|\bar a^h(v)| &\leq C\|h\|_{L^q(\R^3)}^{\frac{q}{3}(\gamma+5)}\|h\|_{L^{\infty}(\R^3)}^{1-(\gamma+5)\frac{q}{3}}.\label{ineq:aloinf}
\end{align}
For any $1 \leq p < \frac{-3}{\gamma + 1}$, there exists $C > 0$ such that
\begin{align}
|\partial_{v_i} \bar a^h (v)| &\leq  C\|h\|_{L^1(\R^3)}^{1 + \frac{p}{3}(\gamma +1)}\|h\|_{L^{\frac{p}{p-1}}(\R^3)}^{-(\gamma+1)p/3}. \label{ineq:agrad}  % &\mbox{if } p> 3/(4+\gamma).
\end{align}
Finally, for any $1 \leq p < \frac{-3}{\gamma}$ and $\gamma \in (-3,-2]$
\begin{align}
\abs{\bar c^h(v)} &\leq  C\|h\|_{L^1(\R^3)}^{1 + \frac{p}{3}\gamma} \|h\|_{L^{\frac{p}{p-1}}(\R^3)}^{-\gamma p/3}, \qquad  \gamma \in (-3,-2]. \label{ineq:c}
\end{align}
\end{lemma}
\begin{proof}
For $s\in (-3,0)$, splitting the integral into $\abs{v_*} \leq R$ and $\abs{v_*} > R$, applying H\"older's inequality on each, we have for $1 \leq p < 3/(3-s) < r \leq \infty$
\begin{align}
\left| \int_{\R^3} h(v-v_*) \abs{v_*}^s \dd v_* \right| \lesssim R^{-s +3 - 3/p} \norm{h}_{L^p} + R^{-s + 3 - 3/r}\norm{h}_{L^r}. \label{ineq:absvstarin}
\end{align}  
Optimizing in $R$ gives the estimates \eqref{ineq:a1hi}, \eqref{ineq:aloinf}, and \eqref{ineq:c} (using also $|\Pi(v_*)| \leq 1$). 
For \eqref{ineq:agrad}, first one integrates by parts and uses $|\partial_{v_i}(\Pi(v_*) |v_*|^{\gamma+2})|\lesssim |v_*|^{\gamma+1}$ before applying again \eqref{ineq:absvstarin}. 
\end{proof}

The next lemma ensures that the formal identity $\int_{\R^3}Q_{L}(g,g) (1 + |w|^2) \dd w = 0$ and entropy dissipation inequality $\int_{\R^3}\log g Q_{L}(g,g) \dd w \leq 0$ are valid under our assumptions on $g$. 

%\textcolor{red}{[By the way, using that the entropy dissipation should be striclty negative unless you are exactly Maxwellian, there may be some way to treat unbounded profiles. There might be a little graduate student project if we think the community would be sufficiently interested]}
%{\color{red}{Maria: should $g$ be defined almost everywhere in $y$ to make the following lemma meaningful ? }}
\begin{lemma}\label{l:invariants}
Let $\chi \in C^\infty(B(0,2))$ be a smooth cut-off function, such that $\chi(|x|) = 1$ for $\abs{x} \leq 1$. With $g \in L_w^1 \cap L_w^\infty(\R^3) \cap C^\infty$, we have
\begin{itemize}
	\item %With $g \in L_w^1 \cap L_w^\infty(\R^3)$ satisfying \eqref{e:g-smooth}, we have  
	$$\lim_{R \to \infty}\int_{\R^3} \chi\left( \frac{w}{R} \right) Q_{L}(g,g) \dd w = 0.$$
	\item If, in addition, $g$ satisfies $|w|^2 g   \in L_w^1 $ we have
	$$\lim_{R \to \infty}\int_{\R^3} \chi\left( \frac{w}{R} \right)|w|^2 Q_{L}(g,g) \dd w = 0.$$
	\item If, in addition, $g$ satisfies 
	\begin{eqnarray}\label{e:entropy}
% \begin{split}
&  g \log g \in L_w^1,  \quad \nabla \sqrt{g} \in L_w^2, %\int_{\R^3} |\nabla \sqrt g(y,w)|^2 \dd w < \infty , \\
% \end{split}
 \end{eqnarray}
then we have
	$$\lim_{R \to \infty}\int_{\R^3} \chi\left( \frac{w}{R} \right) \log g Q_{L}(g,g) \dd w \leq 0.$$
\end{itemize}
\end{lemma}
\begin{proof}
Define 
\begin{align*}
\chi_R(w) : = \chi(w/R);
\end{align*}	
notice that 
\begin{align}
\abs{\nabla^j \chi_R} \lesssim R^{-j}, 	\label{ineq:gradchi}
\end{align}
and moreover, the derivatives are only supported in the region $w \approx R$.
	
Since $Q_L$ can be written as 
\[Q_L(g,g) = \nabla_v \cdot \left(\int_{\R^3} \Pi(v_*) |v_*|^{\gamma+2} [g(v-v_*)\nabla_v g(v) - g(v)\nabla_v g(v-v_*)] \dd v_*  \right),\]
integration by parts gives 
\begin{align*}
\int_{\R^3} \chi_R Q_{L}(g,g) \dd w&  = - \int_{\R^3} \nabla\chi_R\\
& \quad \cdot \left(\int_{\R^3} \Pi(w_*) |w_*|^{\gamma+2} [g(w-w_*)\nabla_w g(w) - g(w)\nabla_w g(w-w_*)] \dd w_*  \right) dw \\ 
& =  2\int_{\R^3} g \nabla \chi  \cdot  \textrm{div}_w \bar a^g (w)\;dw \\ 
& \quad + \int_{\R^3} g \nabla^2 \chi_R : \bar a^g (w)\;dw.%\left(\int_{\R^3} \Pi(v_*) |v_*|^{\gamma+2} [ g(v) g(v-v_*)] \dd v_*  \right) dv. 
\end{align*}
%It suffices to show that the factors involving the $\dd v_*$ convolutions are $L^1$ in $v$. 
Since $g \in L_w^p$ for any $p\ge 1$, thanks to  \eqref{ineq:aloinf} and  \eqref{ineq:agrad} the integrals are bounded uniformly in $R$ and hence as $R\to +\infty$, we have by \eqref{ineq:gradchi}, $\lim_{R \to \infty}\int_{\R^3} \chi\left( \frac{w}{R} \right) Q_{L}(g,g) \dd w = 0$. 

Similarly, integration by parts yields
\begin{align*}
\int_{\R^3} \chi_R |w|^2  Q_{L}(g,g) \dd w =  \;& \frac{2}{R}\int_{\R^3} g |w|^2\nabla \chi_R \cdot  \textrm{div}_w \bar a^g (w)\;dw \\
& +  4 \int_{\R^3} g \chi_R   \textrm{div}_w \bar a^g (w) \cdot w \;dw\\
& +  \frac{1}{R^2} \int_{\R^3} g |w|^2 \nabla^2 \chi : \bar a^g (w)\;dw + 2 \int_{\R^3} g \chi_R   Tr[\bar a^g]\;dw \\
& + \int_{\R^3} g \sum_{i,j} \bar a_{i,j}^g \left( 2 w_j \partial_i \chi_R + 2 w_i \partial_j \chi_R\right) \;dw.
\end{align*}
All of the terms involving derivatives of the cutoff function
vanish as $R \to \infty$ by the same arguments as used in the previous case. 
Since 
$$
2  g    \textrm{div}_w \bar a^g (w) \cdot w +  g  Tr[\bar a^g] =  2\frac{g(w)}{|z-w|} \nabla_z g(z) \cdot w + \frac{g(w)g(z)}{|w-z|},
$$
integration by parts yields 
\begin{align*}
\int_{\R^3}2  g   \chi_R \left( \textrm{div}_w \bar a^g (w) \cdot w +  g  Tr[\bar a^g] \right) \;dw & = \int_{\R^6}\chi_R\frac{g(w)g(z)}{|w-z|} \left[ \frac{1}{|z-w|} - \frac{|z-w|^2}{|z-w|^3}\right] \;dzdw \\ 
& \quad - 2\int_{\R^3} \nabla \chi_R \cdot w g \bar{a}^g \; dw \\ 
& =  - 2\int_{\R^3} \nabla \chi_R \cdot w g \bar{a}^g \; dw. 
\end{align*}
Hence, by the assumptions on $g$  and \eqref{ineq:gradchi}, we can pass to the limit $R\to +\infty$ and get 
\begin{align*}
\lim_{R \to \infty}\int_{\R^3} \chi\left( \frac{w}{R} \right) |w|^2 Q_{L}(g,g) \dd w =  0.
\end{align*}

For the entropy inequality, we begin with 
\begin{align*}
\int_{\R^3} \chi_R \log g  Q_{L}(g,g) \dd w = \;& \int_{\R^3} (2g\ln g -g) \nabla \chi_R \cdot  \textrm{div}_w \bar a^g (w)\;dw \\% \left(\int_{\R^3} \Pi(v_*) |v_*|^{\gamma+2} [ g(v)\nabla_{v_*} g(v-v_*)] \dd v_*  \right) dv \\ 
& + \int_{\R^3} (g\ln g -g) \nabla^2 \chi_R : \bar a^g (w)\;dw \\
& - \int_{\R^3} \chi_R \left [ \left \langle   \bar a^g (w) \frac{ \nabla g }{\sqrt{g}}, \frac{ \nabla g }{\sqrt{g}}\right \rangle - \nabla g \cdot  \textrm{div}_w \bar a^g (w)\right] \;dw\\
= \;& \int_{\R^3} (2g\ln g -2g) \nabla \chi_R \cdot  \textrm{div}_w \bar a^g (w)\;dw \\% \left(\int_{\R^3} \Pi(v_*) |v_*|^{\gamma+2} [ g(v)\nabla_{v_*} g(v-v_*)] \dd v_*  \right) dv \\ 
& + \int_{\R^3} (g\ln g -g) \nabla^2 \chi_R : \bar a^g (w)\;dw \\
& - \int_{\R^3} \chi_R \left [ \left \langle   \bar a^g (w) \frac{ \nabla g }{\sqrt{g}}, \frac{ \nabla g }{\sqrt{g}}\right \rangle - g \bar c^g(w) \right] \;dw,
\end{align*}
using the identity
\begin{align*}
 \int_{\R^3} \chi_R \nabla g \cdot  \textrm{div}_w \bar a^g (w) \;dw = & -\int_{\R^3} g \nabla \chi_R \cdot  \textrm{div}_w \bar a^g (w)\;dw \\
&  + \int_{\R^3} \chi_R g \bar c^g(w)  \;dw.
\end{align*}
Using, once more,  \eqref{ineq:aloinf}, \eqref{ineq:agrad}, $g \in L^p$ for any $p\ge 1$, \eqref{ineq:gradchi} and this time also (\ref {e:entropy}), we conclude that the first two integrals vanish as $R\to +\infty$.
Moreover, 
$$
\left \langle   \bar a^g (w) \frac{ \nabla g }{\sqrt{g}}, \frac{ \nabla g }{\sqrt{g}}\right \rangle - g \bar c^g(w)
$$
is  a $L^1$ function, thanks to \eqref{ineq:aloinf} and (\ref {e:entropy}) for the first term, and  (\ref{ineq:c}) for the second.  Hence, dominated convergence theorem allows to pass to the limit 
$$
\lim_{R \to \infty} \int_{\R^3} \chi_R \log g  Q_{L}(g,g) \dd w =  - \int_{\R^3} \left [ \left \langle   \bar a^g (w) \frac{ \nabla g }{\sqrt{g}}, \frac{ \nabla g }{\sqrt{g}}\right \rangle - g \bar c^g(w) \right] \;dw.
$$
The thesis follows by noticing that the integral on the right hand side can be rewritten as 
\begin{align*}
\int_{\R^3} & \left [ \left \langle   \bar a^g (w) \frac{ \nabla g }{\sqrt{g}}, \frac{ \nabla g }{\sqrt{g}}\right \rangle - g \bar c^g(w) \right] \;dw \\
&= \int_{\R^3}\int_{\R^3}\frac{g (w) g (w_*)}{|w-w_*|^{-\gamma-2}} \left\langle  \Pi(w-w_*)\left( \frac{\nabla g}{g}-\frac{\nabla_* g}{g_*}\right), \left( \frac{\nabla g}{g}-\frac{\nabla_* g}{g_*}\right) \right \rangle\;dvdv_* \ge 0.
\end{align*}
\end{proof}

\subsection{Proof of Theorem \ref{t:landau}}

As first step, we plug ansatz \eqref{e:ansatzII} into \eqref{e:mainL} and change to self-similar variables $y$ and $w$. The left-hand side of \eqref{e:mainL} transforms as follows: 
\begin{align*}
\partial_t f + v\cdot \nabla_x f &= \partial_t \phi + v\cdot \nabla_x \phi + \frac 1 {(-t)^{2+\theta(3+\gamma)}} g + \frac {1} {(-t)^{1+\theta(3+\gamma)}}\left( \frac{\partial y}{\partial t}\cdot \nabla_y g +   \frac{\partial w}{\partial t}\cdot\nabla_w g +  v \cdot \nabla_x g\right)\\
&= \partial_t \phi + v\cdot \nabla_x \phi + \frac 1 {(-t)^{2+\theta(3+\gamma)}} \left(g +  (1+\theta)y\cdot \nabla_y g +   \theta  w\cdot\nabla_w g + w\cdot \nabla_y g\right).
\end{align*}
Here, and throughout the proof, all terms involving $g$ are evaluated at $(y,w)$, and terms involving $\phi$ are evaluated at $(t,(-t)^{1+\theta}y,(-t)^\theta w)$, unless otherwise noted. 
Moreover, we have
\begin{align*}
\bar a^f &= \bar a^\phi + \frac 1 {(-t)^{1+\theta(3+\gamma)}} \int_{\R^3} \Pi(v-v_*) |v-v_*|^{\gamma+2}  g\left(\frac x {(-t)^{1+\theta}}, \frac {v_*} {(-t)^\theta}\right) \dd v_*\\
&= \bar a^\phi + (-t)^{-\gamma\theta-1} \int_{\R^3} \Pi((-t)^\theta(w- w_*)) |(-t)^\theta(w- w_*)|^{\gamma+2} g(y,w_*) \dd w_*\\
&= \bar a^\phi + (-t)^{2\theta-1} \int_{\R^3} \Pi(w- w_*) |w- w_*|^{\gamma+2} g(y,w_*) \dd w_*\\
 &= \bar a^\phi + (-t)^{2\theta-1} \bar a^g,
\end{align*}
and, for $\gamma > -3$, 
\begin{align*}
\bar c^f &= \bar c^\phi + \frac 1 {(-t)^{1+\theta(3+\gamma)}} c_\gamma\int_{\R^3} |v-v_*|^\gamma g\left(\frac x {(-t)^{1+\theta}}, \frac {v_*} {(-t)^\theta}\right) \dd v_*\\
&= \bar c^\phi + \frac 1 {(-t)} \bar c^g.
\end{align*}
%Clearly, $\bar c^f = \bar c^\phi + (-t)^{-1} \bar c^g$ also holds in the case $\gamma = -3$, where we have $\bar c^g = g$. We also have 
Taking into account that $D_v^2 f = D_v^2 \phi + (-t)^{-(1+(5+\gamma)\theta)} D_w^2 g$, the right-hand side of \eqref{e:mainL} becomes
\begin{align*}
Q_L(f,f) &= Q_L(\phi,\phi) + \frac 1 {(-t)^{1+(5+\gamma)\theta}}\tr(\bar a^\phi D_w^2 g) + (-t)^{2\theta-1}\tr (\bar a^g D_v^2 \phi) + \frac 1 {(-t)^{1+\theta(3+\gamma)}} \bar c^\phi g\\
&\quad + \frac 1 {(-t)} \bar c^g \phi + \frac 1 {(-t)^{2+(3+\gamma)\theta}} \tr(\bar a^g D_w^2 g)  + \frac 1 {(-t)^{2+\theta(3+\gamma)}}\bar c^g g.
\end{align*}
%\begin{align*}
%Q(f,f) &= \frac 1 {(-t)^2} \left[ \tr(\bar a^\psi D_w^2 \psi + \tr(\bar a^\psi D_w^2 g) + \tr (\bar a^g D_w^2 \psi) + 2 \psi g + \tr(\bar a^g D_w^2 g)  +  g^2\right].
%\end{align*}
Multiplying through by $(-t)^{2+\theta(3+\gamma)}$ and rearranging the terms, we have
\begin{equation}\label{e:expansion-L}
\begin{split}
0 = & g + (1+\theta)y\cdot \nabla_y g + \theta w\cdot \nabla_w g + w\cdot \nabla_y g - Q_{L,w}(g,g)\\
& - (-t)^{1-2\theta} \tr(\bar a^\phi D_w^2 g) - (-t)\bar c^\phi g - (-t)^{1+\theta(3+\gamma)}\bar c^g \phi  \\
&  - (-t)^{1+\theta(5+\gamma)} \tr(\bar a^g D_v^2\phi)  + (-t)^{2+\theta(3+\gamma)} (\partial_t \phi + v\cdot \nabla_x \phi - Q_L(\phi,\phi)),
\end{split}
\end{equation}
where 
$$
Q_{L,w} = \tr(\bar a^g D_w^2 g) + \bar c^g g.
$$ 
As $t\to 0$, we expect that the terms $g + (1+\theta)y\cdot \nabla_y g + \theta w\cdot \nabla_w g + w\cdot \nabla_y g - Q_{L,w}(g,g)$.  will dominate. To show that, we will prove in the next two lemmas that the two error functions
\[\begin{split}
\mathcal E_1(\phi,g)  &: = - (-t)^{1-2\theta} \tr(\bar a^\phi D_w^2 g) - (-t)\bar c^\phi g  - (-t)^{1+\theta(3+\gamma)}\bar c^g \phi\\
\mathcal{E}_2(\phi,g) &: =  - (-t)^{1+\theta(5+\gamma)} \tr(\bar a^g D_v^2\phi)  + (-t)^{2+\theta(3+\gamma)} (\partial_t \phi + v\cdot \nabla_x \phi - Q_L(\phi,\phi)),
\end{split}
\]
are decaying to zero as $t\to 0$. More precisely
\begin{lemma}\label{lemma_E_1}
For all $R> 0$ we have 
\begin{equation}\label{e:E1}
	\lim_{t \to 0} \;\sup_{\abs{y} \leq R} \;\sup_{v \in \R^3}  \abs{\mathcal E_1} = 0. 
\end{equation}
\end{lemma}
\begin{proof}
Using \eqref{ineq:aloinf} and assumption \eqref{e:g-smooth}, for all $1 \leq p < 3/(\gamma + 5)$, we estimate the first term in $\mathcal E_1$ as
%By \eqref{ineq:aloinf} and assumption \eqref{e:g-smooth}, for all $1 \leq p < 3/(\gamma + 5)$ there holds  
\begin{align}
|(-t)^{1-2\theta}\tr(\bar a^\phi D_w^2g)| & \lesssim (-t)^{1-2\theta} \norm{\phi}_{L^p_v}^{\frac{p}{3}(\gamma + 5)}\norm{\phi}_{L^\infty_v}^{1- \frac{p}{3}(\gamma + 5)}  \nonumber \\ 
& = \left( (-t)^{1 + \theta(\gamma+ 3) - \frac{3\theta}{p}} \norm{\phi}_{L^p_v}\right)^{\frac{p}{3}(\gamma + 5)} \left( (-t)^{1 + \theta(\gamma+ 3)} \norm{\phi}_{L^\infty_v}\right)^{1- \frac{p}{3}(\gamma + 5)}.\label{3.11}
\end{align}
The  second factor vanishes for any $\gamma$ and $\theta$, thanks to  (\ref{case2}) with $p=\infty$. For the first term, we still use (\ref{case2}) if $ p > \frac{3\theta}{1+\theta(3+\gamma)}$.  Hence, we need  
\begin{align*}
	\frac{3\theta}{1 + \theta(3+\gamma)} < \frac{3}{\gamma + 5},% \label{ineq:struc}
\end{align*}
which is fulfilled if $\theta < 1/2$. 

Now let us turn to the term involving $c^\phi$. 
We have for any $1 \leq p < 3/(3+\gamma)$ by \eqref{ineq:c}, 
\begin{align}\label{c_inf}
%\abs{\bar c^\phi} & \lesssim \norm{\phi}_{L^\infty} \quad \textup{ if } \quad \gamma=-3
\abs{\bar c^\phi} & \lesssim \norm{\phi}_{L^p_v}^{\frac{p}{3}(\gamma+3)} \norm{\phi}_{L^\infty_v}^{1- {\frac{p}{3}(\gamma+3)}}. %\quad \textup{ if } \quad \gamma\in (-3,-2]. 
\end{align}
Hence, by assumption \eqref{e:g-smooth}, 
\begin{align*}
(-t) \abs{\bar c^\phi g} \lesssim \left((-t)^{1 + \theta(\gamma +3) - \frac{3\theta}{p}}\norm{\phi}_{L^p_v}\right)^{\frac{p}{3}(\gamma+3)} \left((-t)^{1 + \theta(\gamma +3)}\norm{\phi}_{L^\infty_v}\right)^{1- {\frac{p}{3}(\gamma+3)}}. 
\end{align*}
Analogous to above, the  second factor vanishes for any $\gamma$ and $\theta$, thanks to  (\ref{case2}) with $p=\infty$. For the first one, we still use (\ref{case2}) if $ p > \frac{3\theta}{1+\theta(3+\gamma)}$.  Hence, we need once more 
%in order to ensure that we can find a suitable $p$, we obtain the structural requirement
\begin{align*}
	\frac{3\theta}{1 + \theta(3+\gamma)} < \frac{3}{\gamma + 3}, 
\end{align*}
which is satisfied  for any $\gamma$ and $\theta$. 

%Hence the term $(-t) \bar c^\phi g$ vanishes as in \eqref{e:E1}. 

Finally, by assumption \eqref{e:g-smooth} and \eqref{ineq:c} for $g$, we get  
\begin{align*}
(-t)^{1+\theta(3+\gamma)}\abs{\bar c^g \phi} \lesssim (-t)^{1+\theta(3+\gamma)}\abs{\phi},  
\end{align*}
which vanishes  for any $\gamma$ and $\theta$, thanks again to  (\ref{case2}) with $p=\infty$. 
This completes the proof of the lemma. 
\end{proof}

\begin{lemma}\label{lemma_E_2}
For any $R_1,R_2> 0$, we have. 
\begin{equation}\label{e:E2}
	\lim_{t \to 0}\; \sup_{\abs{y} \leq R_1}\;  \sup_{\abs{w} \leq R_2} \;\abs{\mathcal E_2} = 0. 
\end{equation}
\end{lemma}
\begin{proof}
First, as $\abs{c^g} + \abs{\bar{a}^g} \lesssim 1$ by assumption \eqref{e:g-smooth}, 
\begin{align*}
	\abs{\mathcal E_2} & \lesssim  (-t)^{1+\theta(3+\gamma)}\abs{\phi} +  (-t)^{1+\theta(5+\gamma)}\abs{D_v^2\phi}  \\ 
	& \quad + (-t)^{1+\theta(3+\gamma) + (1+\theta)} \abs{ (-t)^{-\theta}\partial_t \phi + w\cdot \nabla_x \phi} +  |(-t)^{2+\theta(3+\gamma)} Q_L(\phi,\phi)|.
\end{align*}
The first  three terms on the right hand side converge to zero in $L^\infty_{loc}( \R^6)$ directly by assumption \eqref{e:lim-phi-t1}. 
We now look at the collision term 
\[ (-t)^{2+\theta(3+\gamma)}Q_L(\phi,\phi) = (-t)^{2+\theta(3+\gamma)} [\tr(\bar a^\phi D_v^2\phi)+\bar c^\phi \phi] .\] 
Analogously as in the previous lemma,  we write %$$
%\|\bar a^\phi \|_{L^\infty_v(\mathbb{R}{^3})} \le c(\gamma,p) \|\phi  \|_{L^p_v(\mathbb{R}{^3})}^{(\gamma + 5)\frac{p}{3}} \|\phi  \|_{L^\infty_v(\mathbb{R}{^3})}^{1-(\gamma + 5)\frac{p}{3}} ,
%$$
%for $1 \leq p < 3/(\gamma + 5)$ to be chosen below, which yields 
\begin{align*}
	|(-t)^{2 + \theta(3+\gamma)} \tr(\bar a^\phi D_v^2\phi)|  \lesssim &\left( (-t)^{1 + \theta(3+\gamma) - 3\theta/p}\|\phi  \|_{L^p_v(\mathbb{R}{^3})} \right)^{(\gamma + 5)\frac{p}{3}} \\
	&\cdot\left( (-t)^{1 + \theta(3+\gamma)}\|\phi  \|_{L^\infty_v(\mathbb{R}{^3})}\right)^{1-(\gamma + 5)\frac{p}{3}}  \left( (-t)^{1 + \theta(3+\gamma) + 2\theta}\| D^2_v \phi  \|_{L^\infty_v(\mathbb{R}{^3})} \right).  
\end{align*}
The second and third factor vanish for any $\gamma$ and $\theta$, thanks to  (\ref{case2}) with $p=\infty$. The first term is identical to the one in (\ref{3.11}) and vanishes for $p<\frac{3}{5+\gamma}$ and $\theta < \frac{1}{2}$. 
To estimate the last term $\bar c^\phi \phi$ we use (\ref{c_inf}) and get, 
\begin{align*}
(-t)^{2+\theta(3+\gamma)} \abs{\bar c^\phi \phi} & \lesssim (-t)^{2+\theta(3+\gamma)} \norm{\phi}_{L^p}^{\frac{p}{3}(\gamma+3)} \norm{\phi}_{L^\infty}^{2 - \frac{p}{3}(\gamma+3)} \\ 
 & \lesssim \left( (-t)^{1 + \theta (3+\gamma) - \frac{3\theta}{p}} \norm{\phi}_{L^p}\right)^{\frac{p}{3}(\gamma+3)} \left( (-t)^{1 + \theta (3+\gamma)} \norm{\phi}_{L^\infty}\right)^{2 - \frac{p}{3}(\gamma+3)}, 
\end{align*}
with $1 \leq p < 3/(\gamma+3)$. The second factor vanishes for any $\gamma$ and $\theta$, thanks to  (\ref{case2}) with $p=\infty$. For the first term we still use  (\ref{case2}) with $\frac{3\theta}{1+\theta(3+\gamma)} < p < \frac{3}{(\gamma+3)}$. This completes the proof of the lemma. 
\end{proof}

Having shown that the dominant terms in (\ref{e:expansion-L}) are 
\begin{equation*}
g + (1+\theta)y\cdot \nabla_y g + \theta w\cdot \nabla_w g + w\cdot \nabla_y g - Q_{L,w}(g,g),
\end{equation*}
our next step is to show that the only admissible solution to 
$$
0 = g + (1+\theta)y\cdot \nabla_y g + \theta w\cdot \nabla_w g + w\cdot \nabla_y g - Q_{L,w}(g,g),
$$
is the trivial one $g \equiv 0$. 
\begin{proof}[Proof of Theorem \ref{t:landau}]
We  multiply \eqref{e:expansion-L} by a general smooth test function $\psi(y,w)$ with compact support in $\R^6$, and take the limit $t\to 0$. Thanks to  \eqref{e:E1} and \eqref{e:E2}, we conclude that $g$ satisfies
\begin{equation}\label{e:yw}
g + (1+\theta)y\cdot \nabla_y g + \theta w\cdot \nabla_w g + w\cdot \nabla_y g - Q_{L,w}(g,g) = 0,
\end{equation}
 in the sense of distributions. However, using the regularity and decay assumptions for $g$  \eqref{e:g-smooth}, we conclude that \eqref{e:yw} holds pointwise for all $(y,w)\in \R^6$. 
 
 The rest of the theorem is devoted to  showing that the only solution to (\ref{e:yw}) is the trivial one, $g \equiv 0$. The proof varies depending on the value of $\theta$. We distinguish three cases, $\theta \neq \pm \frac{1}{3}$, $\theta = 1/3$ and $\theta =- 1/3$.
 
 \begin{itemize}
 \item Let { \em{$\theta \neq  1/3$,  $\theta \neq -1/3$.}} Let $\chi_R(w) = \chi(w/R)$ be a cutoff function and $\varphi \in C^\infty_0(B(0,1))$ a smooth function such that $\int_{\R^3} \varphi(y) dy = 1$. 
 For $y_0 \in \R^3$ define $\varphi_{y_0}(y) := \varphi(y+y_0)$.  We multiply (\ref{e:yw}) by $\chi_R(w) \varphi(y + y_0)$ for some $y_0 \in \R^3$ and integrate in $\R^6$. Recall the decomposition $g = h(y,w) + q(w)$. Integration by parts yields 
   \begin{align*}
	(1-3\theta)\int_{\R^6} \chi_R g \varphi_{y_0}  \dd w \dd y & - \theta \int_{\R^6}  g \varphi_{y_0}  w \cdot \nabla_w \chi_R \;dw \;dy - (1+\theta)\int_{\R^6} h \chi_R  y\cdot \nabla_y \varphi_{y_0} \dd w  \dd y \\
	&-  \int_{\R^6} h\chi_R  w\cdot \nabla_y \varphi_{y_0} \dd w \dd y -3(1+\theta)  \int_{\R^6} h\chi_R  \varphi_{y_0}   \dd w \dd y \\
	= & \int_{\R^6} \varphi  \chi_R Q_{L,w}(g,g) \dd w \dd y .
\end{align*}
 We first take the limit $R\to +\infty$. 
 Note that $\abs{w \cdot \nabla_w \chi_R} \lesssim 1$ by \eqref{ineq:gradchi} and $w \cdot \nabla_w \chi_R$ converges to zero pointwise. Therefore, by the dominated convergence theorem, the second term vanishes. For the collision term, we use Lemma  \ref{l:invariants}.
 The remaining terms converge by the dominated convergence theorem and the assumptions on $h$ and $q$.  
 Therefore, we obtain
  \begin{align*}
	(1-3\theta)\int_{\R^6} (q+h) \varphi_{y_0}  \dd w \dd y& - 3(1+\theta)  \int_{\R^6} h \varphi_{y_0}  \dd w \dd y \\
	& - \int_{\R^6} h   w\cdot \nabla_y \varphi_{y_0} \dd w \dd y- (1+\theta)\int_{\R^6} h   y\cdot \nabla_y \varphi_{y_0} \dd w  \dd y =0.
 \end{align*}
Next, we perform the limit $y_0 \to \infty$. Thanks to the assumption {{$(1+\abs{y} + \abs{w})h \in L^1(\R^6)$}}, all of the terms involving $h$ vanish as $y_0 \to \infty$ by the dominated convergence theorem. Hence, the above identity reduces to 
\begin{align*}
	(1-3\theta)\int_{\R^3} q(w) \dd w  = 0. 
\end{align*}
Since $q \geq 0$, we conclude that $q \equiv 0$. The condition $g\ge 0$ and $q=0$ implies $h\ge 0$. We now go back to (\ref{e:yw}) with $q=0$, multiply it by  $\chi_{R_1}(y)  \chi_{R_2}(w)$ with $\chi_{R_1}(y) = \chi(y/R_1)$ and $\chi_{R_2}(w) = \chi(w/R_2)$ and integrate in $\R^6$. Similarly as above, we first take the limit $R_2 \to +\infty$ and get 
\begin{eqnarray*}%\label{e:intwy}
	-2(1+3\theta)\int_{\R^6} \chi_{R_1} h  \dd w \dd y & - (1+\theta)\int_{\R^6} h y\cdot \nabla_y \chi_{R_1} \dd w \dd y  \\
	&- \int_{\R^6} h w\cdot \nabla_y \chi_{R_1} \dd w\dd y = 0.
\end{eqnarray*}
Using the assumption that {{$(1+w)h \in L^1(\R^6)$}}, we can pass to the limit $R_1 \to +\infty$ in the above equation by the dominated convergence theorem as we used above (in particular that $y \cdot \nabla_y \chi_{R_1}$ is uniformly bounded and converges pointwise to zero) and obtain 
$$
-2(1+3\theta)\int_{\R^6}  h  \dd w \dd y =0,
$$
which implies $h \equiv 0$.

% If $\theta = -1/3$, we multiply (\ref{e:intw}) with $q=0$ by $ \chi_R(y)\ln h$ and integrate. After several integration by parts we get 
%\begin{align*}
%-\frac{2}{3} \int_{\R^6} h \ln h y \cdot \nabla \chi_R \dd w \dd y & -  \int_{\R^6}  h \ln h w \cdot \nabla \chi_R \dd w \dd y \\
%&+ 2 \int_{\R^6} h \chi_R  \dd w \dd y + \frac{2}{3}  \int_{\R^6} h y \cdot \nabla \chi_R \dd w \dd y \\
%&+  \int_{\R^6} h w \cdot \nabla \chi_R  \dd w \dd y  \le 0.
%\end{align*}
%Assuming {\color{blue}{$(1+y+w)h \in L^1(\R^6)$ and $(y+w)h \ln h \in L^1(\R^6)$}}, we can pass to the limit in $R$ and obtain 
%$$
%2 \int_{\R^6} h   \dd w \dd y \le 0,
%$$
%which implies again $h \equiv 0$. 

 \item Let {\em{$\theta = 1/3.$} } As before, let $\chi_R(w) = \chi(w/R)$ a cutoff function and $\varphi \in C^\infty_0(B(0,1))$ a smooth function such that $\int_{\R^3} \varphi(y) dy = 1$ and take $\varphi_{y_0}(y) = \varphi(y + y_0)$. This time we multiply (\ref{e:yw}) by $\chi_R(w) |w|^2 \varphi_{y_0}$ for some $y_0 \in \R^3$ and integrate in $\R^6$. We obtain 
 \begin{align*}
 -\frac{2}{3} \int_{\R^6}   g \varphi_{y_0} |w|^2 \chi_R \;dwdy &- \frac{1}{3} \int_{\R^6}   g \varphi_{y_0} |w|^2 w \cdot \nabla_w \chi_R \;dwdy  \\
 -2  \int_{\R^6} h \chi_R \varphi_{y_0} |w|^2 \;dwdy &- \frac{2}{3}  \int_{\R^6} |w|^2 h \chi_R y \cdot \nabla_y \varphi_{y_0} \;dydw   \\
 -  \int_{\R^6} |w|^2 \chi_R h w \cdot \nabla_y \varphi_{y_0} \;dwdy &=  \int_{\R^6} \varphi_{y_0} \chi_R(w) |w|^2  Q_{L}(g,g)\;dwdy.
 \end{align*}
Thanks to the condition {{$q(1+ |w|^2) \in L^1_w$ and $h(1 + \abs{y}\abs{w}^2 + |w|^3)  \in L^1(\R^6)$ }} and Lemma  \ref{l:invariants} for the collision term, we can pass to the limit $R\to +\infty$ using the dominated convergence theorem as above and we get 
\begin{align*}
 -\frac{2}{3} \int_{\R^6}   (q+h) \varphi_{y_0} |w|^2  \;dwdy & -4  \int_{\R^6} h  \varphi_{y_0} |w|^2 \;dwdy  \; + \\
 - \frac{4}{3}  \int_{\R^6} |w|^2 h  y \cdot \nabla_y \varphi_{y_0} \;dydw  & -  \int_{\R^6} |w|^2  h w \cdot \nabla_y \varphi_{y_0} \;dwdy = 0.
 \end{align*}
 The limit $y_0 \to +\infty$, using again $h(1 + \abs{y}\abs{w}^2 + |w|^3)  \in L^1_{y,w}$, gives 
 $$
 -\frac{2}{3} \int_{\R^3}   q|w|^2  \;dw =0,
 $$
which implies $q\equiv 0$. To show that also $h \equiv 0$, we multiply (\ref{e:yw}) with $q=0$ by $\chi_{R_1}(w)\chi_{R_2}(y) |w|^2$ and integrate in $\R^6$. After taking the limit $R_1 \to +\infty$ we obtain 
\begin{align*}
-\frac{2}{3} \int_{\R^6}   h \chi_{R_2}|w|^2  \;dwdy & -4  \int_{\R^6} h  \chi_{R_2} |w|^2 \;dwdy  \; + \\
 - \frac{4}{3}  \int_{\R^6} |w|^2 h  y \cdot \nabla_y\chi_{R_2} \;dydw  & -  \int_{\R^6} |w|^2  h w \cdot \nabla_y \chi_{R_2}\;dwdy = 0.
\end{align*}
The limit $R_2 \to +\infty$ yields 
$$
-\frac{14}{3} \int_{\R^6}   h |w|^2  \;dwdy =0,
$$
which implies, since $h\ge 0$, that  $h \equiv 0$. 

\item Let {\em{$\theta = -1/3.$} } Mimicking the same calculation of the case $\theta =1/3$, we multiply (\ref{e:yw}) by $\chi_R(w) |w|^2 \varphi_{y_0}$, integrate over $\R^6$ and perform the limit $R \to +\infty$. We get \begin{align*}
 \frac{8}{3} \int_{\R^6}   (q+h) \varphi_{y_0} |w|^2  \;dwdy & -2  \int_{\R^6} h  \varphi_{y_0} |w|^2 \;dwdy  \; + \\
 - \frac{2}{3}  \int_{\R^6} |w|^2 h  y \cdot \nabla_y \varphi_{y_0} \;dydw  & -  \int_{\R^6} |w|^2  h w \cdot \nabla_y \varphi_{y_0} \;dwdy = 0.
 \end{align*}
The limit $y_0 \to +\infty$, using again {{$h(1+|w|^3)  \in L^1(\R^6)$}}, gives 
 $$
 \frac{8}{3} \int_{\R^3}   q|w|^2  \;dw =0,
 $$
which implies $q\equiv 0$. To show that also $h \equiv 0$, we multiply (\ref{e:yw}) with $q=0$ by $\chi_{R_1}(w)\chi_{R_2}(y) |w|^2$ and integrate in $\R^6$. The limits $R_1, R_2 \to +\infty$ yield 
$$
\frac{2}{3} \int_{\R^6}   h |w|^2  \;dwdy =0,
$$
which implies, since $h\ge 0$, that  $h \equiv 0$. 
\end{itemize}
This finishes the proof of the theorem. 
\end{proof}

\section{The Vlasov-Poisson-Landau system}  \label{s:landau-poisson}
In this section we analyze the following system:
\[
\partial_t f + v\cdot \nabla_x f + F[f]\cdot \nabla_v f =Q_L(f,f),
\]
with 
\[
F[f]=C\int_{\R^3} \frac{x-z}{|x-z|^3}\left[\int_{\R^3} f(z,v) \dd v - n_0 (z)\right] \dd z,
\]
where $n_0(x)  \geq 0$ is a fixed function that models a neutralizing background.  
If $C \ge 0$ we are in the repulsive interaction case, if $C\le 0$ we are in the attractive case.

Unlike Landau and Boltzmann, the Vlasov-Poisson-Landau equation only has a one-parameter scaling symmetry, and hence there is only one case to consider: $\gamma = -3$ and $\theta = -\frac 1 3$. 
Therefore, our ansatz becomes
\begin{equation}\label{e:ansatz-LCP}
f(t,x,v) = \phi(x,t,v) + \frac{1}{(-t)}g\left(\frac{x}{(-t)^{2/3}}, (-t)^{1/3}v\right).
\end{equation}
For the analysis of the Vlasov-Landau-Poisson system, our proof requires the use of $\ln g$ as a test function, which requires the following additional assumptions on the profile 
\begin{eqnarray}\label{e:sec_moment}
	(1 + \abs{w})(g \ln g - g) \in L^1_{y,w},  \quad \nabla \sqrt{g} \in L^\infty_y L_w^2, \quad g\in L^1_{y,w}.
\end{eqnarray}
One new detail must be addressed: due to the non-locality in $x$ introduced by the interaction term, we must be more specific about the global structure of the solution. 
Our methods can handle any of the following three cases, each of which is physically relevant: 
\begin{itemize}
	\item[(a)] $n_0 = 0$ and $f \in L^1(\R^6)$. This case is natural for studying gravitational interactions (where $f$ models the density of stars or galaxies and hence in the attractive case).  
	\item[(b)] The physical domain is $\mathbb{T}^3_x$ and we take $n_0(x) =n_0 = \frac{1}{( 2\pi )^3} \int_{\R^6} f(x,v) dx dv$. 
	This case is most natural for studying periodic perturbations arising in the kinetic theory of plasmas (where $f$ will model the density of electrons in a plasma and the $n_0$ models a homogeneous background of ions, hence the interactions are repulsive). 
	\item[(c)] The solution is given by $f(t,x,v) = \mu(v) + h(t,x,v)$ where $\mu$ is a Maxwellian with fixed density, momentum, and temperature, $h \in L^1(\R^6)$ with average zero, and $n_0 = \int_{\R^3} \mu(v) dv$. This case is most natural for studying localized perturbations of a homogeneous plasma (here $f$ models the density of electrons in a plasma and the $n_0$ models a homogeneous background of ions, hence the interactions are repulsive). 
\end{itemize}
Our proof easily adapts to any of these three cases, so we focus on the simplest one, which is case (a). It is straightforward to extend the argument to cases (b) and (c). 
As in the previous section, we will assume $\phi$ satisfies \eqref{e:lim-phi-t1}, which, for $i=0$ and $\theta=\frac{1}{\gamma}=-\frac{1}{3}$, reads
\begin{equation}\label{e:lim-phi-2}
\lim_{t \to 0} (-t)^{1+  \tfrac{1}{p} + \tfrac 2 3\ell -\tfrac 1 3 j} \sup_{\abs{y} \leq R} \norm{D_v^j D_x^\ell \phi(t,(-t)^{2/3}y,\cdot)}_{L^p_v} = 0,
\end{equation}
for all $1 \leq p \leq \infty$, $0\leq j\leq 1$, $0\leq \ell \leq 2$, $R>0$. 
Due to the nonlocality in $x$ of the interaction force, we also enforce the condition that the density $\rho_\phi = \int_{\R^3} \phi(t,x,v) dv$ satisfies 
%\begin{equation}\label{e:lim-rho}
% \norm{\rho_\phi(t)}_{L^p_x} , \norm{\rho_\phi(t)}_{L^\infty_x} < +\infty,
%\end{equation}
%and the same estimate 
\begin{equation}\label{e:lim-rho}
	\lim_{t \to 0} (-t)^{1+  \tfrac{1}{p}} \norm{\rho_\phi(t)}_{L^p_x} = 0,
\end{equation}
for some $p <3$ and also for $p=\infty$ (and hence everything in between by interpolation). 

%{\color{blue}{The conditions on $g$ are similar to the previous section, however, due to the non-locality, we currently can only treat the case that $wg \in L^1(\R^6)$.}} 
\begin{theorem}\label{main_VLP}
	Let $f$ satisfy \eqref{e:f-condition} and \eqref{e:singularity}, $\phi$ satisfy  \eqref{e:lim-phi-t1} and \eqref{e:lim-rho}, and $g$ satisfy \eqref{e:g-smooth} and \eqref{e:sec_moment}. %Let, moreover, $g$ satisfy the entropy conditions \eqref{e:entropy}. 
	Then for any solution to the Vlasov Landau Poisson system of the form
	\begin{equation*}%\label{e:ansatz-LCP}
f(t,x,v) = \phi(x,t,v) + \frac{1}{(-t)}g\left(\frac{x}{(-t)^{2/3}}, (-t)^{1/3}v\right),
\end{equation*}
 we must have $g\equiv 0$.
\end{theorem}

\begin{proof}
Define the self similar variables 
$$
y:= \frac{x}{(-t)^{2/3}}, \quad w := v(-t)^{1/3}.
$$
A brief computation shows that $F[f]$ transforms as
\[ 
F\left[\phi + \frac{1}{(-t)}g\right]\left( (-t)^{2/3}y \right) = F[\phi ]\left( (-t)^{2/3} y \right)  + \frac{1}{(-t)^{4/3}}F[g]\left( y \right) .
\]
%{\color{blue}{The integrals in $F[\phi]$ and $F[g]$ are well-defined since both $\phi(t,\cdot,\cdot)$ and $g$ lie in $L^1(\R^6)$. }}
%From now on, we write $y = x/(-t)^{2/3}$ and $w= (-t)^{1/3}v$ and analogous to above, $F[\phi]$ will always be evaluated at $(-t)^{2/3} w$. 
We now plug \eqref{e:ansatz-LCP} into the system \eqref{e:LCP}. The resulting equation, after multiplying by $(-t)^2$, reads as 
 \begin{align*}
&(-t)^2[\partial_t\phi + v\cdot \nabla_x \phi + F[\phi]\cdot \nabla_v \phi - Q_L(\phi,\phi)] \\
&+  {(-t)^{4/3}}  F[\phi]\cdot \nabla_w g + (-t)^{2/3} F[g]\cdot \nabla_v \phi  \\
& - (-t) [Q_{L,w}(\phi,g) + Q_{L,w}(g,\phi)]\\
&+  g + \frac{2}{3}y \cdot \nabla_y g -\frac{1}{3} w \cdot \nabla_w g  +  w \cdot \nabla_y g + F[g]\cdot \nabla_w g - Q_{L,w}(g,g) = 0.
\end{align*}
We now define the error as
\[\begin{split}
\mathcal E(\phi,g) :=  \;& (-t)^2[\partial_t\phi + v\cdot \nabla_x \phi + F[\phi]\cdot \nabla_v \phi - Q_L(\phi,\phi)] +  (-t)^{4/3}  F[\phi]\cdot \nabla_w g\\
& + (-t)^{2/3} F[g]\cdot \nabla_v \phi - (-t) [Q_{L,w}(\phi,g) + Q_{L,w}(g,\phi)].
\end{split} \]
We claim $\mathcal E(\phi,g) \to 0$ as $t\to 0$, uniformly on compact sets of $R^3_y\times \R^3_w$.  All the terms, except the ones with $F[\cdot]$, appeared already in $\mathcal E_1$ and $\mathcal E_2$ and converge to zero, as proven in Lemma \ref{lemma_E_1} and \ref{lemma_E_2}.  We start by analyzing 
 $$(-t)^2F[\phi]\cdot \nabla_v \phi. $$ We have 
\begin{align*}
 \lim_{t \to 0}  \sup_{\abs{w, y} \leq R}  |(-t)^2F[\phi]\cdot \nabla_v \phi | \le  \lim_{t \to 0}  \;\sup_{\abs{y} \leq R} (-t)^{4/3}|F[\phi]| (-t)^{2/3} \| \nabla_v \phi\|_{L^\infty_v} =0,
\end{align*} 
thanks to  \eqref{e:lim-phi-2} with $p = \infty$, $\ell = 0$, $j=1$. 
Note that $\sup_{\abs{y} \leq R} \;|F[g]| $ is bounded thanks to our assumption $g\in L^\infty_{y,loc}L^1_w$ and $g\in L^1_{y,w}$ and therefore we similarly have 
\begin{align*}
 \lim_{t \to 0}  \sup_{\abs{w, y} \leq R}  |(-t)^{2/3}F[g]\cdot \nabla_v \phi | \le  \lim_{t \to 0}  \sup_{\abs{y} \leq R} \;|F[g]| (-t)^{2/3} \| \nabla_v \phi\|_{L^\infty_v} =0.
\end{align*} 
We turn next to the term $(-t)^{4/3}F[\phi] \cdot \nabla_w g$. For this, we use the interpolation \eqref{ineq:absvstarin} with $s=-2$, $r = \infty$, and $1 \leq p < 3$ to obtain 
\begin{align*}
(-t)^{4/3}\abs{F[\phi]} & \lesssim (-t)^{4/3}\norm{\rho_{\phi}}_{L^p}^{p/3} \norm{\rho_{\phi}}_{L^\infty}^{1-p/3} \\ 
& \lesssim \left( (-t)^{1 + \frac{1}{p}} \norm{\rho_{\phi}}_{L^p} \right)^{p/3} \left( (-t)\norm{\rho_{\phi}}_{L^\infty} \right)^{1-p/3}, 
\end{align*}
and, hence, the associated term vanishes by \eqref{e:lim-rho}.

Thus, in the limit as $t\to 0$, we obtain
\begin{align}\label{g_LCP}
g + \frac 2 3 y \cdot \nabla_y g - \frac 1 3 w \cdot \nabla_w g  +  w \cdot \nabla_y g + F[g]\cdot \nabla_w g  = Q(g,g) .
\end{align}
From where, we multiply by $ \chi_{R_1}(w) \chi_{R_2}(y) \log g$ and integrate in both variables; after integration by parts we get 
\begin{align*}
\int_{\R^6} g \chi_{R_2}(y) \chi_{R_1}(w) \;dwdy& -\frac{2}{3} \int_{\R^6}  (g \ln g - g)\chi_{R_1}(w) y \cdot \nabla_y \chi_{R_2}(y) \;dwdy \\
+ \frac{1}{3} \int_{\R^6}  (g \ln g - g)\chi_{R_2}(y) w \cdot \nabla_w \chi_{R_1}(w) \;dwdy &- \int_{\R^6}  (g \ln g - g)\chi_{R_1}(w) w \cdot \nabla_y \chi_{R_2}(y) \;dwdy\\
&- \int_{\R^6}  (g \ln g - g)\chi_{R_2}(y)  F[g] \cdot  \nabla_w \chi_{R_1}(w)\;dwdy \\
&= \int_{\R^6}\chi_{R_2}(y) \chi_{R_1}(w)  Q_{L,w}(g,g)\;dwdy.
\end{align*}
With the assumptions on $g$, we can pass to the limit $R_1\to +\infty$ by the dominated convergence theorem and get 
\begin{align*}
\int_{\R^6} g \chi_{R_2}(y)  \;dwdy& -\frac{2}{3} \int_{\R^6}  (g \ln g - g)y \cdot \nabla_y \chi_{R_2}(y) \;dwdy \\
 &- \int_{\R^6}  (g \ln g - g) w \cdot \nabla_y \chi_{R_2}(y) \;dwdy\le 0,
\end{align*}
using Lemma \ref{l:invariants} for the right hand side. Thanks to the assumption
\begin{align*}
&{{(1+ w)(g \ln g - g) \in L^1_{w,y}}}, 
\end{align*}
we take the limit $R_2\to +\infty$ and obtain%. Passing to the limit $R \to \infty$ as in Lemma \ref{l:invariants} 
%(here using \eqref{e:entropy}) implies the result
\[
\int_{\R^6}  g \dd w\dd y \leq 0,
\]
which implies $g\equiv 0$.
\end{proof}

%{\color{blue}{This computation is not needed: We now multiply (\ref{g_LCP}) by $|w|^2$ and integrate over $\mathbb{R}^3 \times \mathbb{R}^3$. We get:
%\begin{align}\label{eq:g_II}
%(-2-8\theta) \int\int g |w|^2 \;dydw -8C\pi \int V \textrm{div}_y\int gw \;dwdy =0.
%\end{align}
%Next we multiply  (\ref{g_LCP}) by $V$ and integrate by parts: 
%\begin{align*}
%\int V \textrm{div}_y\int gw \;dwdy = (3\theta -1) \int \int gV \;dydw -(1+\theta) \int \int V y\cdot \nabla_y g\;dwdy.
%\end{align*}
%Since 
%$$
%\int \Delta V y \cdot \nabla_y V\;dy = -\frac{3}{2} \int |\nabla V|^2\;dy =- \frac{3}{2} \int \int gV\;dwdy,
%$$
%integration by parts yields 
%$$
%\int \int V y\cdot \nabla_y g\;dwdy =-3  \int \int gV\;dwdy + \int \Delta V y \cdot \nabla_y V\;dy = - \frac{9}{2} \int \int gV\;dwdy,
%$$
%which implies 
%\begin{align}\label{VC}
%\int V \textrm{div}_y\int gw \;dwdy = \left( (3\theta -1) + \frac{9}{2}(1+\theta)\right) \int \int gV \;dydw.
%\end{align}
%Plugging (\ref{VC}) into (\ref{eq:g_II}) we get 
%\begin{align*}
%(-2-8\theta) \int\int g |w|^2 \;dydw -8C\pi\left( (3\theta -1) + \frac{9}{2}(1+\theta)\right) \int \int gV \;dydw=0,
%\end{align*}
%which for $\theta = -\frac{1}{3}$ becomes 
%$$
%\frac{2}{3} \int\int g |w|^2 \;dydw  -8C\pi  \int \int gV \;dydw=0.
%$$
%If $C\le 0$ the above equation is satisfied if and only if $g=0$. This means we can rule out blow up for $\theta = -\frac{1}{3}$ for the gravitational case.  }}
%

\section{The Boltzmann equation} \label{s:boltzmann}

We recall the Boltzmann equation 
\begin{equation*}%\label{e:mainB}
\partial_t f + v\cdot \nabla_x f = Q_B(f,f) := \int_{\R^3} \int_{\mathbb S^{2}} B(v-v_*, \sigma) [f(v_*')f(v') - f(v_*)f(v)]\dd \sigma \dd v_*.
\end{equation*}
The velocities are related by the formulas
\begin{align}
v' &= \frac{v+v_*} 2 + \frac{|v-v_*|} 2 \sigma ,\\
v_*' &= \frac{v+v_*} 2 - \frac{|v-v_*|} 2 \sigma,
\end{align}
and the pre-post collisional angle $\eta$ (usually denoted $\theta$ in the literature) is defined by 
\[\cos \eta = \left\langle \frac{v-v_*}{|v-v_*|}, \sigma\right\rangle.\]
We take the standard non-cutoff collision kernel described by
\[ B(v-v_*,\sigma) :=  |v-v_*|^\gamma b(\cos \eta),\]
 for some $\gamma \in (-3,1]$, with the angular cross-section $b$ satisfying the asymptotics
\begin{align}
%B(v-v_*,\sigma) &= |v-v_*|^\gamma b(\cos \eta),\\ 
%\cos \eta &= \left\langle \frac{v-v_*}{|v-v_*|}, \sigma\right\rangle,\\
b(\cos\eta) &\approx \eta^{-2-2s} \quad \mbox{ as } \eta\to 0,
%v' &= \frac{v+v_*} 2 + \frac{|v-v_*|} 2 \sigma ,\\
%v_*' &= \frac{v+v_*} 2 - \frac{|v-v_*|} 2 \sigma,
\end{align}
for some $s\in (0,1)$. We assume $\gamma + 2s < 0$ for ease of presentation. Results similar to Theorem \ref{t:boltzmann} should also be available when $\gamma + 2s \geq 0$. 

%For the Boltzmann equation \eqref{e:mainB}, we have the same conclusion as Theorem \ref{t:landau}, under similar hypotheses. We assume $\gamma + 2s < 0$ for ease of presentation. A similar result should hold in the case $\gamma + 2s \geq 0$.

As mentioned above, the Boltzmann equation obeys the same family of scaling laws as the Landau equation, so the approximately self-semilar ansatz \eqref{e:ansatz} takes the same form.

In our main result for the Boltzmann equation, we derive the same conclusion as Theorem \ref{t:landau}, under similar hypotheses:

\begin{theorem}\label{t:boltzmann}
Let $\gamma > -3$ and $s\in (0,1)$ be such that $\gamma+2s < 0$, and assume $-1< \theta < 1/(2s)$.  Let $f$ be a smooth solution of the Boltzmann equation \eqref{e:mainB} that satisfies \eqref{e:f-condition} and \eqref{e:singularity}. 
Assume that $\phi$ satisfies \eqref{e:lim-phi-t1}. 
For $g$, assume it satisfies  \eqref{e:g-smooth} as well as $(1+|w|^{2+\gamma}) g(y,\cdot) \in L^1_w(\R^3)$ for all $y\in \R^3$, and that there exist $h$ and $q$ such that 
\begin{align*}
g(y,w) = q(w) + h(y,w),
\end{align*}
with 
$$(1+ \abs{y} + |w|)h\in L_{y,w}^1(\R^6)\quad \textrm{and} \quad q \in L_w^1(\R^3).$$ % and $h$  and $q$ both satisfy \eqref{e:g-smooth}. \\
Finally,  if $\theta = \pm1/3$ we additionally assume that 
$$(1+ \abs{y}\abs{w}^2 + |w|^3) h \in L^1_{y,w}(\R^6) \quad \textrm{and} \quad (1+ |w|^2) q \in L^1_w(\R^3).$$
	%\item[(b)] if $\theta = 1/3$, then $q$ satisfies the entropy conditions \eqref{e:entropy}.
%\end{enumerate}
Then, for any solution to the Boltzmann equation \eqref{e:mainB} of the form 
\begin{equation*}
f(t,x,v) = \phi(t,x,v) + \frac 1 {(-t)^{1+\theta(3+\gamma)}}\, g\left( \frac x {(-t)^{1+\theta}}, \frac v {(-t)^{\theta}}\right),
\end{equation*}
 we must have $g\equiv 0$.% and hence no approximate self-similar singularity of this type can occur.  
\end{theorem}
\begin{remark}
As for Landau, if $g \in L^1_{y,w}$, then it suffices to assume $(1+ |w|)g \in L^1_{y,w}$ (and $(1+ |w|^3) g \in L^1_{y,w}(\R^6)$). 
\end{remark}

%\begin{theorem}\label{t:boltzmann}
%	Let $\gamma > -3$ and $s\in (0,1)$ be such that $\gamma+2s < 0$, and assume $-1< \theta < 1/(2s)$. Assume that $f$, $\phi$, and $g$ satisfy \eqref{e:f-condition} and \eqref{e:singularity}. Assume that $\phi$ satisfies \eqref{e:lim-phi-t1}, and that $g$ satisfies \eqref{e:g-smooth}. 
%	For $g$, assume either that
%	\begin{enumerate}
%		\item[(a)] $g\in L^1(\R^6)$, and if $\theta = -1/3$,  \eqref{e:entropy2} holds.
%		\item[(b)] $g\in L^1_w(\R^3)$ is independent of $y$, and if $\theta = 1/3$, \eqref{e:entropy2} holds.
%	\end{enumerate} 
%	Then, for any solution to the Boltzmann equation \eqref{e:mainB} of the form \eqref{e:ansatz}, we must have $g\equiv 0$.
%
%\end{theorem}
 
 Specializing to the homogeneous case as above, we have the following result:
\begin{theorem}\label{theom:Boltzmann_Hom}
 With $\gamma$ and $s$ as in Theorem \ref{t:boltzmann}, and $\frac 1 {|\gamma|} < \theta < \frac 1 {2s}$, assume that $f = f(v,t)$ has finite mass and second moment and satisfies 
\begin{align*}%\label{e:singularity}
 f \in C^\infty((-T,0) \times  \R^3), f \in C^\infty((-T,0] \times \R^3 ).
\end{align*}%\eqref{e:singularity}.
Assume that $\phi$ satisfies \eqref{e:lim-phi-t1}, and that $g = g(v)$ is such that 
$$
g\in C^{\infty}(\R^3), \quad (1+|w|^{2+\gamma}) g \in L^1_w.%, \quad D^j_{w,loc} g \in L^p_w,
$$
%for all $p\ge 1$ and $0\le j \le 2$. 
If $\theta = 1/3$, then additionally assume that $g$ satisfies $ (1+|w|^2)g \in L^1_w$. 

Then, for any solution to the homogeneous Boltzmann equation of the form 
\begin{equation*}
f(t,v) = \phi(t,v) + \frac 1 {(-t)^{1+\theta(3+\gamma)}}\, g\left( \frac v {(-t)^{\theta}}\right),
\end{equation*}
we must have $g \equiv 0$, and hence no approximate self-similar singularity of this type can occur. 
\end{theorem}

To prove Theorem \ref{t:boltzmann}, we need the following decomposition of the collision operator $Q_B(f_1,f_2)$ into two terms: by adding and subtracting $f_1(v_*')f_2(v)$ inside the integral, we write $Q_B(f_1,f_2) = Q_1(f_1,f_2) + Q_2(f_1,f_2)$, with
\begin{equation}\label{e:decomposition}
\begin{split}
Q_1(f_1,f_2) &= p.v. \int_{\R^3} \int_{\mathbb S^2} b(\cos\eta) |v-v_*|^\gamma(f_2(v') - f_2(v)) f_1(v_*') \dd \sigma \dd v_*,\\
Q_2(f_1,f_2) &= f_2(v)\int_{\R^3} \int_{\mathbb S^2} b(\cos\eta) |v-v_*|^\gamma(f_1(v_*') - f_1(v_*)) \dd \sigma \dd v_*,
\end{split}
\end{equation}
for functions $f_1$ and $f_2$ defined on $\R^3$. %This decomposition was first introduced in \cite{villani2002review} (Chapter 6.2) and, later, has been used extensively in the conditional regularity program which began in  \cite{silvestre2016boltzmann} and continued in \cite{imbert2016weak, imbert2018decay, cameron2019boltzmann, imbert2019smooth}. 
The term $Q_1(f_1,f_2)$ acts as a fractional differential operator of order $2s$, and can roughly be thought of as analogous to the term $\tr(\bar a^{f_1} D_v^2 f_2)$ from the Landau collision operator. The following lemma, quoted from \cite{silvestre2016boltzmann}, makes this point of view clearer:
\begin{lemma}{\cite[Section 4]{silvestre2016boltzmann}}\label{l:Q1}
There holds
\[ Q_1(f_1,f_2) = \int_{\R^3} [f_2(v+h) - f_2(v)]K_{f_1}(v,h) \dd h,\]
where
\[K_{f_1}(v,h) \approx |h|^{-3-2s} \int_{\{z:z\cdot h = 0\}} f_2(v+z) |z|^{\gamma+2s+1} \dd z,\]
with implied constants depending only on $\gamma$, $s$, and the angular cross-section $b$. The kernel $K_{f_1}$ is symmetric ($K_{f_1}(v,-h) = K_{f_1}(v,h)$) and satisfies the following bounds: for any $r>0$,
\[ \int_{B_{2r}\setminus B_r} K_{f_1} (v,h) \dd h \leq C\left(\int_{\R^3} |z|^{\gamma+2s} f_1(v+z) \dd z \right) r^{-2s},\]
whenever the right-hand side is finite.
\end{lemma}
Estimating the convolution as in the proof of Lemma \ref{l:coeffs}, we see that Lemma \ref{l:Q1} implies
\begin{equation}\label{e:kernel-bound}
% \int_{B_{2r}\setminus B_r} K_{f_1} (v,h) \dd h \leq C \|f_1\|_{L^1_v}^{(\gamma+2s+3)/3} \|f_1\|_{L^\infty_v}^{-(\gamma+2s)/3} r^{-2s},
\int_{B_{2r}\setminus B_r} K_{f_1} (v,h) \dd h \leq C \|f_1\|_{L^p_v}^{(\gamma+3+2s)p/3} \|f_1\|_{L^{\infty}_v}^{1-(\gamma+3+2s)p/3} r^{-2s},
\end{equation}
for all $r>0$ and $1\leq p < 3/(\gamma+3+2s)$. For $f_1$ depending on $(x,v)$ or $(t,x,v)$, we will write $K_{f_1}(x,v,h)$ or $K_{f_1}(t,x,v,h)$. Note that the ``p.v.'' in $Q_1$ is only necessary when $s>1/2$. We omit the ``p.v'' from now on, since our functions are smooth enough ($C^2$ in $v$) that the value of the integral is well-defined.

%\begin{lemma}{\cite[Lemma 2.3]{imbert2019lowerbounds}}\label{l:C2}
%	If $f_1:\R^3\to \R_+$ satisfies 
%	\[ \int_{\R^3} f_1(v_*) |v-v_*|^{\gamma+2s} \dd v_* < \Lambda,\]
%	and $f_2\in C^2_v(\R^3)$, then
%	\[ |Q_{1}(f_1,f_2)| \lesssim \Lambda \|f_2\|_{L^\infty(\R^3)}^{1-s}\|D_v^2 f_2\|_{L^\infty(\R^3)}^s.\]
%\end{lemma}

For $Q_2(f_1,f_2)$, symmetry effects for grazing collisions (see \cite{alexandre2000entropy}) imply the following representation: %will use upper bounds that take advantage of symmetry effects for grazing collisions ($\eta \approx 0$), which is by now a well-understood phenomenon, see \cite{alexandre2000entropy, villani1999boltzmann, silvestre2016boltzmann}. The following estimate, taken from \cite{silvestre2016boltzmann}, is in a convenient form for our proof:
\begin{lemma}{\cite[Lemmas 5.1 and 5.2]{silvestre2016boltzmann}}\label{l:Q2}
The integral in $Q_2(f_1,f_2)$ satisfies
\[ Q_2(f_1,f_2) = f_2(v) \int_{\R^3} f_1(v+z) (C|z|^{\gamma}) \dd z,\]
where the constant $C$ depends only on $\gamma$ and $s$.
\end{lemma}
%\begin{proof}
%By \cite{silvestre2016boltzmann}, Lemmas 5.1 and 5.2, $Q_2(f_1,f_2)$ can be written
%\[ Q_2(f_1,f_2) = f_2(v) \int_{\R^3} f_1(v+z) \tilde B(|z|) \dd z,\]
%with $\tilde B(|z|) = C|z|^{\gamma}$, for some $C$ depending on $\gamma$ and $s$. A straightforward estimate of the integral $\int f_1(v+z) |z|^\gamma \dd z$ implies the conclusion of the lemma.
%%\[ \int_{\R^3} f_1(v+z)|z|^\gamma \dd z \lesssim  \]
%\end{proof}
In other words, surprisingly, $Q_2(f_1,f_2)$ is equal up to a constant to $\bar c^{f_1} f_2$ in the notation of the Landau equation.

Lemmas \ref{l:Q1} and \ref{l:Q2} imply in particular that $Q_B(f_1,f_2)$ is well-defined in a pointwise sense whenever $f_1 \in L^1_v \cap L^\infty_v$ and $f_2 \in L^\infty_v \cap C^2_v$. As in the previous sections (see Lemma \ref{l:invariants}), we need to use a form of the identities $\int Q_B(g,g) \dd w = \int |w|^2 Q_B(g,g) \dd w = 0$:
%This inequality follows from the pre-post collisional change of variables $(w,w_*)\mapsto (w',w_*')$ applied to the formula for $Q_B$ seen in \eqref{e:mainB}. 
\begin{lemma}\label{l:Binvariants}
With $\chi\in C^\infty(B(0,2))$ a smooth-cutoff with $\chi(|x|) = 1$ for $|x|\leq 1$, and with $g\in L^1_x\cap L^\infty_w (\R^3)$ satisfying \eqref{e:g-smooth} as well as $(1+|w|^{2+\gamma}) g \in L^1$, we have
\[ \lim_{R\to \infty} \int_{\R^3} \chi\left(\frac w R\right) Q_B(g,g) \dd w = 0,\]
%If, in addition, $g$ satisfies $|w|^2 g \in L^1_w$, then we have
and
\[ \lim_{R\to \infty} \int_{\R^3} \chi\left(\frac w R\right)|w|^2 Q_B(g,g) \dd w = 0.\]
\end{lemma}
This lemma is more or less understood in the literature on the Boltzmann equation. We give a proof for the convenience of the reader, and because we could not find an easy reference to apply in our setting.
\begin{proof}
The well-known weak formulation of the Boltzmann collision operator allows one to make sense of integrals of the form $\int_{\R^3} \varphi Q_B (g,g) \dd w$ using smoothness of $\varphi$. For any function $f$, let us introduce the abbreviations $f = f(w)$, $f_* = f(w_*)$, $f' = f(w')$, and $f_*' = f(w_*')$. Applying the pre-post collisional change of variables $(\sigma, w,w_*)\leftrightarrow (\sigma,w',w_*')$ (with unit Jacobian) one has
\begin{equation*}%\label{e:weak-form}
 \int_{\R^3} \varphi Q_B(g,g) \dd w = \int_{\R^3} \int_{\R^3}\int_{\mathbb S^2} B(w-w_*, \sigma) g g_* [\varphi' - \varphi] \dd \sigma \dd w_* \dd w.
 \end{equation*}
Symmetrizing further with the change of variables $w \leftrightarrow w_*$, which also exchanges $w'$ and $w_*'$, one has 
\begin{equation}\label{e:weak-form}
 \int_{\R^3} \varphi Q_B(g,g) \dd w = \frac 1 2 \int_{\R^3} \int_{\R^3} \int_{\mathbb S^2} B(w-w_*,\sigma) g g_* [\varphi_*' + \varphi' - \varphi_* - \varphi] \dd \sigma \dd w_* \dd w.
 \end{equation}
These formal calculations can be justified rigorously under our assumption that $g$ is smooth, provided that $\varphi$ is (say) $C^2$ and compactly supported.

%With $\varphi(w) = |w|^2 \chi(w/R)$, the integrand on the right converges to 0 pointwise because $|w_*'|^2 + |w'|^2 = |w_*|^2 + |w|^2$ (this identity reflects the conservation of energy during collisions). 
The expression $\varphi_*' + \varphi' - \varphi_* - \varphi$ is equal to zero for the following three cases: $\varphi = 1$, $\varphi = w$, and $\varphi = |w|^2$. This reflects the conservation of mass, momentum, and energy during collisions. 

Taylor expanding $\varphi$ and using $w_*' + w' = w_* + w$ and $|w_*' - w_*| = |w' - w|$, we have
\[ \begin{split}
\varphi_*' + \varphi' - \varphi_* - \varphi  &= \nabla \varphi_* \cdot (w_*' - w_*) + \nabla \varphi \cdot (w' - w) + O(\|D^2\varphi\|_{L^\infty} |w' - w|^2)\\
&= (\nabla \varphi - \nabla \varphi_* )\cdot (w' - w) +   O(\|D^2\varphi\|_{L^\infty}|w' - w|^2).
%&= O(M|w'-w||w-w_*|) + O(M|w'-w|^2).
\end{split}\]
It follows from the geometry of collisions that $|w'-w| \approx |w-w_*|\eta$. Therefore, the second term in the last expression is proportional to $\eta^2|w-w_*|^2$, which is good enough to cancel the angular singularity $\eta^{-2-2s}$, but the first term is only proportional to $\eta |w-w_*|^2$. We get around this problem in a standard way, by parametrizing $\mathbb S^2$ in spherical coordinates $(\phi,\eta) \in [0,2\pi]\times [0,\pi]$ (where $\eta = 0$ corresponds to $w = w'$) and realizing that $\left|\int_0^{2\pi} (w' - w) \dd \phi\right| \lesssim |w-w_*|\eta^2$. We now have
\[ \left|\int_0^{2\pi}\left[\varphi_*' + \varphi' - \varphi_* - \varphi\right]\dd \phi\right| \lesssim \|D^2\varphi\|_{L^\infty} \eta^2 |w-w_*|^2,\]
 and
 \begin{equation}
 \begin{split}
 \left|  \int_{\R^3} \varphi Q_B(g,g) \dd w \right| &\lesssim \int_{\R^3}\int_{\R^3} g g_* |w-w_*|^{\gamma+2} \int_0^\pi \eta^{-2-2s}\|D^2\varphi\|_{L^\infty} \eta^2 \sin \eta \dd \eta \dd w_* \dd w\\
   &\lesssim \|D^2\varphi\|_{L^\infty} \int_{\R^3}\int_{\R^3} g g_* |w-w_*|^{\gamma+2} \dd w_* \dd w.
   \end{split}
\end{equation}
The last integral is convergent by our assumption that $(1+|w|^{\gamma+2}) g \in L^1$. 

Now, with the choice $\varphi(w) = \chi(|w|/R)$, since $\|D^2\varphi\|_{L^\infty} \lesssim R^{-2}$, we see directly that $\int \chi(|w|/R) Q_B(g,g) \dd w \to 0$ as $R\to \infty$.

If we choose $\varphi(w) = |w|^2\chi(|w|/R)$, then $\|D^2\varphi\|_{L^\infty}$ is bounded independently of $R$. Writing \eqref{e:weak-form} as the integral over $\R^3\times\R^3 \times [0,\pi]$ of 
\[F_R(w,w_*,\eta) := |w-w_*|^\gamma b(\cos\eta) g g_* \int_0^{2\pi}[ \varphi_*' + \varphi' - \varphi_* - \varphi] \dd \phi,\]
then $F_R$ converges to 0 pointwise as $R\to \infty$, and by the above integrability estimates, we may apply Dominated Convergence to conclude $\int |w|^2 \chi(|w|/R) Q_B(g,g) \dd w \to 0$. 
\end{proof}
%To make the proof rigorous, one must side-step the grazing collisions singularity in $Q_B$ by excluding a ball of size $\eps$ near $\eta \approx 0$ and take $\eps \to 0$. We omit the details, but note that the principal value integrals are always well-defined because $g$ satisfies \eqref{e:g-smooth}.

% In the distinguished cases $\theta = \pm 1/3$, we as usual need the entropy inequality
%\[ \int_{\R^3} \log g \, Q_B(g,g) \dd w \leq 0,\]
%which follows from well-known formal calculations (see \cite[Chapter 1, Section 2.4]{villani2002review}) that are valid because $g\log g \in L^1_w$ for all $y$, by our assumption \eqref{e:entropy2}.

Now we are ready to prove our main result for the Boltzmann equation.

\begin{proof}[Proof of Theorem \ref{t:boltzmann}]
Proceeding as in the proof of Theorem \ref{t:landau}, we plug the ansatz \eqref{e:ansatz} into the Boltzmann equation \eqref{e:mainB}, and change variables to $y$ and $w$. The left-hand side transforms in the same way as before. For the right-hand side,
\begin{align*}
Q_B(f,f) = Q_B(\phi,\phi) + \frac 1 {(-t)^{1+\theta(3+\gamma)}}[Q_B(\phi,g) + Q_B(g,\phi)] + \frac 1 {(-t)^{2+2\theta(3+\gamma)}} Q_B(g,g),
\end{align*}
where $g$ is evaluated at $(x/(-t)^{1+\theta}, v/(-t)^\theta)$. Applying the decomposition $Q_B = Q_1+Q_2$ and changing variables appropriately, we have
\[ \begin{split}
Q_1(\phi,g) &\approx \int_{\R^3} [g((v+h)/(-t)^\theta) - g(v/(-t)^\theta)] |h|^{-3-2s} \int_{z\perp h} \phi(v+z) |z|^{\gamma+2s+1} \dd z \dd h\\
&= (-t)^{-2s\theta} \int_{\R^3} [g(w+\tilde h) - g(w)] |\tilde h|^{-3-2s} \int_{z\perp\tilde h} \phi((-t)^\theta w + z) |z|^{\gamma+2s+1} \dd z \dd \tilde h\\
&=: (-t)^{-2s\theta}  \tilde Q_{1}(\phi,g),
\end{split}
\]
and
\[ \begin{split}
Q_1(g,\phi) &\approx \int_{\R^3} [\phi(v+h) - \phi(v)] |h|^{-3-2s} \int_{z\perp h} g((v+z)/(-t)^\theta) |z|^{\gamma+2s+1} \dd z \dd h\\
&= (-t)^{(\gamma+2s+3)\theta} \int_{\R^3} [\phi((-t)^\theta w+ h) - \phi((-t)^{\theta}w)] | h|^{-3-2s} \int_{\tilde z\perp h} g( w + \tilde z) |\tilde z|^{\gamma+2s+1} \dd \tilde z \dd \tilde h\\
&=: (-t)^{(\gamma+2s+3)\theta} \tilde  Q_{1}(g,\phi).
\end{split}
\]
(Note that $\{z\perp h\}$ is a two-dimensional subspace.) By similar calculations, we have
\[ \begin{split}
Q_1(g,g) &= (-t)^{(\gamma+3)\theta} \int_{\R^3} [g(w+\tilde h) - g(w)]|\tilde h|^{-3-2s} \int_{\tilde z\perp \tilde h} g(w+\tilde z)|\tilde z|^{\gamma+2s+1} \dd \tilde z \dd \tilde h\\
&=: (-t)^{(\gamma+3)\theta} Q_{1,w}(g,g).
\end{split}\]
Since $Q_2(h_1,h_2) \approx \bar c^{h_1}h_2$, calculations from the proof of Theorem \ref{t:landau} imply
\[ Q_2(\phi,g) \approx \bar c^\phi g, \quad Q_2(g,\phi) \approx (-t)^{\theta(\gamma+3)} \bar c^g \phi, \]
and we abuse notation by writing $=$ instead of $\approx$ (which amounts to a change of constants).
%
%
%When we rename $v/(-t)^\theta = w$ and make the change of variables $w_* = v_*/(-t)^{\theta}$ in the integrals, the angular cross-section $b(\cos\eta)$ is unchanged, since $(v-v_*)/|v-v_*| = (w-w_*)/|w-w_*|$. We also define $w' = v'/(-t)^\theta$ and $w_*' = v_*'/(-t)^\theta$.  We now have (suppressing the dependence on $x$ and $y$)
%\begin{align*}
%Q_B(g(v/(-t)^\theta),g(v/(-t)^\theta)) &= \int_{\R^3}\int_{\mathbb S^2} b(\cos\eta) |(-t)^\theta(w-w_*)|^\gamma [g(w_*') g(w') - g(w_*)g(w)]\dd \sigma (-t)^{3\theta} \dd w_*\\
%&= (-t)^{(3+\gamma)\theta} Q_{B,w}(g(w),g(w)),
%\end{align*}
%where $Q_{B,w}$ is the collision operator acting in $w$ variables.
%Similarly,
%\begin{align*}
%Q_B(g(v/(-t)^\theta),\phi(v)) &= (-t)^{(3+\gamma)\theta}\int_{\R^3}\int_{\mathbb S^2} B(w-w_*,\sigma) [g(w_*') \phi\left((-t)^\theta w'\right) - g(w_*)\phi((-t)^\theta w)]\dd \sigma \dd w_*\\
%&= (-t)^{(3+\gamma)\theta} Q_{B,w}(g(w),\phi((-t)^\theta w)) ,\\
%Q_B(\phi(v),g(v/(-t)^\theta)) &= (-t)^{(3+\gamma)\theta}\int_{\R^3}\int_{\mathbb S^2} B(w-w_*,\sigma) [ \phi\left((-t)^\theta w_*'\right)g(w') - \phi((-t)^\theta w_*)g(w)]\dd \sigma \dd w_*\\
%&=: (-t)^{(3+\gamma)\theta} Q_{B,w}(\phi((-t)^\theta w),g(w)),
%\end{align*}
%where $B(w-w_*,\sigma) = b(\cos \eta) |w-w_*|^\gamma$. 
%Therefore, the right-hand side of \eqref{e:mainB} can be written
%\begin{align*}
%Q_B(f,f) &= Q_B(\phi,\phi) + \frac 1 {(-t)} [ Q_{B,w}(\phi((-t)^\theta w),g) +  Q_{B,w}(g,\phi((-t)^\theta w))] + \frac 1 {(-t)^{2+ \theta(3+\gamma)}}Q_{B,w}(g,g).
%\end{align*}
Making these substitutions in the right-hand side of \eqref{e:mainB}, multiplying through by $(-t)^{2+\theta(3+\gamma)}$, and grouping terms, we have
\begin{equation}\label{e:expansion}
\begin{split}
0 &= g + (1+\theta)y\cdot \nabla_y g + \theta w\cdot \nabla_w g + w\cdot \nabla_y g -   Q_{B,w}(g,g)  - (-t)^{1-2s\theta} \tilde Q_{1}(\phi,g)\\
&\quad  - (-t)^{1+(\gamma+2s+3)\theta} \tilde Q_{1}(g,\phi)  - (-t)\bar c^\phi g - (-t)^{1+\theta(3+\gamma)} \bar c^g \phi\\
&\quad  + (-t)^{2+\theta(3+\gamma)} (\partial_t \phi + v\cdot \nabla_x \phi - Q_B(\phi,\phi)),
\end{split}
\end{equation}
where, as above, $\phi$ is understood to stand for $\phi(t,(-t)^{1+\theta}y,(-t)^\theta w)$, and $g = g(y,w)$. We remark that, as $s\to 1$, all exponents in this expansion converge to the exponents of the corresponding terms in \eqref{e:expansion-L} in the proof of Theorem \ref{t:landau}. 

The error is defined as
\[\begin{split}
\mathcal E(\phi,g) &= - (-t)^{1-2s\theta} \tilde Q_{1}(\phi,g) - (-t)^{1+(\gamma+2s+3)\theta}  \tilde Q_{1}(g,\phi)\\
&\quad  - (-t)\bar c^\phi g - (-t)^{1+\theta(3+\gamma)} \bar c^g \phi + (-t)^{2+\theta(3+\gamma)} (\partial_t \phi + v\cdot \nabla_x \phi - Q_1(\phi,\phi) - \bar c^\phi \phi).\end{split}\]
%We will now show $\mathcal E(\phi,g) \to 0$, uniformly on compact sets of $\R^3_y\times \R^3_w$.
We claim that for all $R_1,R_2>0$,
\begin{equation}\label{e:E-Boltz}
\lim_{t \to 0} \sup_{\abs{y} \leq R_1} \sup_{\abs{w}\leq R_2} \abs{\mathcal E(\phi,g)} = 0. 
\end{equation}
First, all terms in $\mathcal E(\phi,g)$ that do not involve $Q_1$ or $\tilde Q_1$ are equal to corresponding terms in the proof of Theorem \ref{t:landau}, so the same arguments (which do not require any restriction on $\theta$ from above) imply convergence to zero in the sense of \eqref{e:E-Boltz} for those terms.

Now we address the singular integral terms.\footnote{The following calculation is similar to the proof of \cite[Lemma 2.3]{imbert2019lowerbounds}.}  For any integer $k$, let $A_k$ denote the annulus $\{2^k \leq |v| < 2^{k+1} \}$. Splitting $\tilde Q_1(\phi,g)$ into integrals over $|\tilde h| <1$ and $|\tilde h| \geq 1$, we have, using \eqref{e:kernel-bound},
\begin{equation}\label{e:annulus1}
\begin{split}
&(-t)^{1-2s\theta}\int_{|\tilde h|\geq 1} [g(w+\tilde h) - g(w)] K_\phi(t,(-t)^{1+\theta}y, (-t)^\theta w, \tilde h) \dd \tilde h\\
& =  (-t)^{1-2s\theta} \sum_{k\geq 0} \int_{A_k} [g(w+\tilde h) - g(w)] K_\phi(t,(-t)^{1+\theta}y, (-t)^\theta w, \tilde h) \dd \tilde h\\
&\lesssim (-t)^{1-2s\theta} \|g\|_{L^\infty_w} \sum_{k\geq 0} 2^{-2sk} \|\phi(t,(-t)^{1+\theta}y, \cdot)\|_{L^p_v}^{(\gamma+2s+3)p/3} \|\phi(t,(-t)^{1+\theta}y,\cdot)\|_{L^\infty_v}^{1-(\gamma+2s+3)p/3}\\
&\lesssim \left[ (-t)^{1+\theta(3+\gamma)-3\theta/p} \|\phi(t,(-t)^{1+\theta}y, \cdot)\|_{L^p_v}\right]^{(\gamma+2s+3)p/3}\\
&\qquad \cdot \left[ (-t)^{1+\theta(3+\gamma)} \|\phi(t,(-t)^{1+\theta}y,\cdot)\|_{L^\infty_v}\right]^{1-(\gamma+2s+3)p/3}.
\end{split}
\end{equation}
with $1\leq p < 3/(\gamma+2s+3)$. The second factor converges to 0 by \eqref{e:lim-phi-t1}. For the first factor, we use \eqref{e:lim-phi-t1} again, which requires $p> 3\theta/(1+\theta(3+\gamma))$. Therefore, an admissible $p$ satisfies
\[ \frac {3\theta}{1+\theta(3+\gamma)} < p < \frac 3 {\gamma+ 2s + 3},\]
which is possible since $\theta < 1/(2s)$.

For the integral over $|\tilde h|< 1$, we write
\[ g(w+\tilde h) - g(w) = \nabla_w g(w)\cdot \tilde h + E(w,\tilde h)|\tilde h|^2,\]
with $|E(w,\tilde h)| \lesssim \|D_w^2 g\|_{L^\infty_w}$. By the symmetry of the kernel $K_\phi$, the term with $\nabla_w g(w)$ vanishes, and we have, with $p$ as in the previous paragraph,
\[\begin{split}
&(-t)^{1-2s\theta} \int_{|\tilde h| < 1} [g(w+\tilde h) - g(w)] K_\phi(t,(-t)^{1+\theta}y, (-t)^\theta w, \tilde h) \dd \tilde h\\
& =  (-t)^{1-2s\theta} \sum_{k<0} \int_{A_k} E(w,\tilde h)|\tilde h|^2 K_\phi(t,(-t)^{1+\theta}y, (-t)^\theta w, \tilde h) \dd \tilde h\\
&\lesssim (-t)^{1-2s\theta} \|D_w^2 g\|_{L^\infty_w} \sum_{k<0} 2^{(2-2s)k} \|\phi(t,(-t)^{1+\theta}y, \cdot)\|_{L^p_v}^{(\gamma+2s+3)p/3} \|\phi(t,(-t)^{1+\theta}y,\cdot)\|_{L^\infty_v}^{1-(\gamma+2s+3)p/3},
\end{split}\]
which converges to zero using \eqref{e:lim-phi-t1}, as above. For $\tilde Q_1(g,\phi)$, we divide the $h$ integral into the annuli $\tilde A_k = \{(-t)^\theta 2^k < |v| \leq (-t)^\theta 2^{k+1}\}$ (which are the same as $A_k$, read in $h$ variables rather than $\tilde h$). By a similar Taylor expansion for $|h|<1$, we have, with $p$ as above,
\[\begin{split}
&(-t)^{1+\theta(\gamma + 2s + 3)}  \left( \sum_{k\geq 0}\int_{\tilde A_k} [\phi((-t)^\theta w + h) - \phi((-t)^\theta w)] K_g(y,w,h) \dd h\right.\\
&\left.\qquad + \sum_{k<0} \int_{\tilde A_k} [\phi((-t)^\theta w + h) - \phi((-t)^\theta w)] K_g(y,w,h) \dd h\right)\\
&\lesssim (-t)^{1+\theta(\gamma + 2s + 3)}\left(  \|\phi(t,(-t)^{1+\theta}y,\cdot)\|_{L^\infty_v} \sum_{k\geq 0} (-t)^{-2s\theta} 2^{-2sk} \right.\\
&\left.\qquad + \|D_v^2 \phi(t,(-t)^{1+\theta}y,\cdot)\|_{L^\infty_v} \sum_{k<0} (-t)^{(2-2s)\theta} 2^{(2-2s)k}\right) \|g\|_{L^1_w}^{(\gamma+2s+3)/3}\|g\|_{L^\infty_w}^{-(\gamma+2s)/3}\\
&\lesssim (-t)^{1+\theta(\gamma+3)}\|\phi(t,(-t)^{1+\theta}y,\cdot)\|_{L^\infty_v} + (-t)^{1+\theta(\gamma+5)}\|D_v^2\phi(t,(-t)^{1+\theta}y,\cdot)\|_{L^\infty_v},
\end{split}\]
which converges to 0 by \eqref{e:lim-phi-t1}. 

For the term $Q_1(\phi,\phi)$, we apply \cite[Lemma 2.3]{imbert2019lowerbounds} directly to obtain, with $p$ as above and  $\phi  = \phi(t,(-t)^{(1+\theta)}y,(-t)^\theta w)$,
\[ \begin{split}
(-t)^{2 + \theta(3+\gamma)} Q_1(\phi,\phi) &\lesssim (-t)^{2 + \theta(3+\gamma)} \|D_v^2\phi\|_{L^\infty_w}^s \|\phi\|_{L^\infty_w}^{1-s} \int_{\R^3} \phi(t,(-t)^{(1+\theta)}y,(-t)^{\theta} w - z)|z|^{\gamma+2s} \dd z\\
&\lesssim  (-t)^{2 + \theta(3+\gamma)}\|D_v^2\phi\|_{L^\infty_v}^s \|\phi\|_{L^\infty_v}^{2-s - (\gamma+2s+3)p/3}\|\phi\|_{L^p_v}^{(\gamma+2s+3)p/3}\\
&= \left( (-t)^{1+\theta(5+\gamma)} \|D_v^2 \phi\|_{L^\infty_v}\right)^s \left( (-t)^{1+\theta(3+\gamma)}\|\phi\|_{L^\infty_v}\right)^{2-s-(\gamma+2s+3)p/3}\\
&\quad \cdot \left( (-t)^{1+\theta(3+\gamma)-3\theta/p}\|\phi\|_{L^p_v}\right)^{(\gamma+2s+3)p/3},
\end{split}\]
which also converges to 0 by \eqref{e:lim-phi-t1}. In the second line, we performed a convolution estimate as in \eqref{e:kernel-bound}. We could apply \eqref{e:lim-phi-t1} because  $1+\theta(3+\gamma) - 3\theta/p \geq 0$.

% a calculation similar to \eqref{e:annulus1} implies
%\[\begin{split}
%&(-t)^{2+\theta(3+\gamma)} \int_{|h|\geq 1} [\phi(t,(-t)^{1+\theta} y, (-t)^{\theta}w+h) - \phi(t,(-t)^{1+\theta}y, (-t)^{\theta} w)] K_\phi(t,(-t)^{1+\theta}y, (-t)^\theta w, h) \dd  h\\
%& =  (-t)^{1-2s\theta} \sum_{k\geq 0} \int_{A_k} [g(w+\tilde h) - g(w)] K_\phi(t,(-t)^{1+\theta}y, (-t)^\theta w, \tilde h) \dd \tilde h\\
%&\lesssim (-t)^{1-2s\theta} \|g\|_{L^\infty_w} \sum_{k\geq 0} 2^{-2sk} \|\phi(t,(-t)^{1+\theta}y, \cdot)\|_{L^1_v}^{(\gamma+2s+3)/3} \|\phi(t,(-t)^{1+\theta}y,\cdot)\|_{L^\infty_v}^{-(\gamma+2s)/3}\\
%&\lesssim \left[ (-t)^{1+\theta(3+\gamma)-3\theta} \|\phi(t,(-t)^{1+\theta}y, \cdot)\|_{L^1_v}\right]^{(\gamma+2s+3)/3} \left[ (-t)^{1+\theta(3+\gamma)} \|\phi(t,(-t)^{1+\theta}y,\cdot)\|_{L^\infty_v}\right]^{-(\gamma+2s)/3},
%\end{split}\]
% 

Multiplying \eqref{e:expansion} by any smooth, compactly supported test function and sending $t\to 0$, we conclude
\begin{equation}\label{e:t-zero}
 g + (1+\theta)y\cdot \nabla_y g + \theta w\cdot \nabla_w g + w\cdot \nabla_y g -   Q_{B,w}(g,g) = 0,
 \end{equation}
in the sense of distributions. As above, the regularity assumptions \eqref{e:g-smooth} for $g$ imply \eqref{e:t-zero} holds pointwise.
 
 From this point on, the proof is the same as for the Landau equation (the proof of Theorem \ref{t:landau}), since $\lim_{R\to\infty}\int_{\R^3}\chi(|w|/R) Q_{B,w}(g,g) \dd w = 0$ holds thanks to Lemma \ref{l:Binvariants}, as well as $\lim_{R\to\infty}\int_{\R^3}\chi(|w|/R) |w|^2 Q_{B,w}(g,g) \dd w =0$ in the distinguished cases $\theta = \pm 1/3$. (We can apply Lemma \ref{l:Binvariants} because $|w|^{2+\gamma} g \in L^1_w$, by assumption.) Applying the same argument as above, we conclude $g\equiv 0$ in all cases.
%, we integrate over $\R^6$, integrate by parts, we conclude 
%\[(-2-6\theta)\int_{\R^3} \int_{\R^3} g \dd w \dd y = 0.\]
%If $\theta \neq -\frac 1 3$, we are finished, and if $\theta = -\frac 1 3$, we multiply \eqref{e:t-zero} by $\log g$ and conclude $g\equiv 0$ using $\int_{\R^3} \int_{\R^3} Q_{B,w}(g,g) \log g \dd w \dd y \leq 0$. The integrations by parts are valid by ?? and ??.
\end{proof}

\bibliographystyle{abbrv}
\bibliography{blowup}

\end{document}